\documentclass[12pt]{article}
\usepackage{amsmath}
\usepackage{amssymb, amsthm}

\newtheorem{Rem}{Remark}
\newtheorem{sats}{Theorem}
\newtheorem{prop}{Proposition}
\newtheorem{lem}{Lemma}
\newtheorem{kor}{Corollary}
\newcommand{\banm}{\begin{anm}}
\newcommand{\eanm}{\end{anm}}

\newcommand{\supp}{{\rm supp}}
\newcommand{\W}{{\mathaccent"7017 W}}

\newcommand{\dist}{{\rm dist}}

\begin{document}
\title{Domain dependence of eigenvalues  of elliptic type operators}
\author{Vladimir Kozlov (Link\"oping)\thanks{The author was supported
by the Swedish Research Council (VR)}}
%\date{}
\maketitle

\vspace{-6mm}

\begin{center}
{\it Department of Mathematics, Link\"oping University,
S--581 83 Link\"oping, Sweden  }\\
E-mail address: vlkoz@mai.liu.se
\end{center}

\bigskip \noindent {\bf Abstract.} The dependence on the domain is studied for the
Dirichlet eigenvalues of an elliptic operator considered in bounded domains. Their
proximity is measured by a norm of the difference of two orthogonal projectors
corresponding to the reference domain and the perturbed one; this allows to compare
domains that have non-smooth boundaries and different topology. The main result is
an asymptotic formula in which the remainder is evaluated in terms of this quantity.
As an application, the stability of eigenvalues is estimated by virtue of integrals
of squares of the gradients of eigenfunctions for elliptic problems in different
domains. It occurs that these stability estimates imply well-known inequalities for
perturbed eigenvalues.

\section{Introduction}

We consider eigenvalues of the Dirichlet problem for an elliptic operator in a
domain in $\Bbb R^n$, $n\geq 2$. Our main aim is to study how these eigenvalues
depend on the domain. The first results concerning this classical problem can be
found presumably in the book \cite{R} by Rayleigh. A general technique was proposed
by Hadamard \cite{H1}, \cite{H}, who studied perturbations of a domain with a smooth
boundary. In his works, the boundary of the perturbed domain $\Omega_\varepsilon$ is
described by the function $\tau=\varepsilon h(x')$, where $\tau$ is the variable
along the normal to the boundary, $h$ is a smooth function on the boundary and
$\varepsilon$ is a small parameter. Hadamard's formula for the perturbed first
eigenvalue $\lambda(\Omega_\varepsilon)$ of the Dirichlet Laplacian is as follows:
\begin{equation}\label{Intr1.1}
\lambda(\Omega_\varepsilon)=\lambda(\Omega_0)+
\varepsilon\int_{\partial\Omega_0}|\partial_\tau\varphi|^2h\,dS+o(\varepsilon),
\end{equation}
where $\varphi$ is the first eigenfunction such that
$||\varphi||_{L^2}=1$ and $dS$ is the surface measure on
$\partial\Omega_0$. In various generalizations of this formula,
perturbations are described by a family of smooth mappings, other
boundary conditions are considered as well as more general elliptic
operators; see \cite{Gar}, \cite{B}, \cite{DZ}, \cite{Kaw},
\cite{He} and references cited therein. Some non-smooth
perturbations of smooth boundaries, that are described by normal
shift functions, were treated in \cite{Gr1}--\cite{Gr3},
\cite{Berg}. On the other hand, there are many problems involving
more general classes of perturbations, namely, non-smooth
perturbations of non-smooth boundaries and perturbations that cannot
be described by a family of isomorphisms.

In \cite{K1}--\cite{KN}, another approach was proposed to studying
the dependence of $\lambda$ on the domain and the operator's
coefficients. It is based on an abstract theorem concerning
perturbation of eigenvalues for operators acting in different
spaces. Here we further develop the approach used in \cite{K1}. In
this paper, the main novelty is the application of a new small
parameter to measure the proximity of two spectral problems. This
parameter is a norm of the difference of orthogonal projectors on
Sobolev spaces consisting of functions given on these domains. In
Sect. \ref{SectFeb15a}, we show that the convergence with respect to
this parameter (it is actually a distance) is equivalent to the
convergence in the sense of Mosco or to $\gamma$-convergence, see
\cite{He} and \cite{B}. The latter type of convergence plays an
important role in proving the existence of solutions to various
shape optimization problems dealing with eigenvalues. Another new
point is a refined estimate of the remainder term. It allows us to
extend essentially the class of domains and their perturbations for
which a certain "generalized" asymptotic formula is still valid. Let
us turn to a detailed description of the results. For this purpose
we shall use an example of second order elliptic operator.

Let $\Omega_1$ and $\Omega_2$ be two bounded domains with nonempty intersection. We
consider a bilinear form
\begin{equation}\label{OperA}
(u,v) = \sum_{i,j=1}^n \int_{\Bbb R^n}A_{ij}(x)\partial_{x_j} u \,
\partial_{x_i} \overline{v} \, d x\, ,
\end{equation}
where $A_{ij}$ are bounded measurable real-valued functions such that
$A_{ij}=A_{ji}$ and
\begin{equation}\label{Jan3a}
\nu |\xi|^2\leq\sum_{i,j=1}^nA_{ij}(x)\xi_i\xi_j\leq \nu^{-1}
|\xi|^2
\end{equation}
for all $\xi\in\Bbb R^n\setminus{\cal O}$ and $x\in\Bbb R^n$. The form
$(\cdot,\cdot)$ defines a new inner product in the space $H_k=\W^{1,2}(\Omega_k)$,
$k=1,2$ (we suppose that functions belonging to both of these spaces are extended by
zero to the whole $\Bbb R^n$); the corresponding norm will be denoted by $||\cdot
||$. Let us consider the following spectral problems
\begin{equation}\label{TTN1}
(\varphi,v)=\lambda \langle \varphi,v\rangle\,\;\;\;\mbox{for all
$v\in H_1$}
\end{equation}
and
\begin{equation}\label{TTN2}
(U,V)=\mu \langle U,V\rangle\,\;\;\;\mbox{for all $V\in H_2$,}
\end{equation}
where $\langle\cdot ,\cdot \rangle$ is the inner product in $L^2$,
$\varphi\in H_1$ and $U\in H_2$. By $\lambda_m$ we denote the $m$th
eigenvalue of problem (\ref{TTN1}); let $X_m$ be the corresponding
eigenspace, $J_m=\dim X_m$. Our aim is to describe eigenvalues of
problem (\ref{TTN2}) located near the eigenvalue $\lambda_m$ of
(\ref{TTN1}).

If $\Omega_1$ and $\Omega_2$ are sub-domains of a bounded domain
$D$, then by $S_j$, $j=1,2,$ we denote the orthogonal projector
defined on $\W^{1,2}(D)$ whose image belongs to $H_j$. For
characterizing the proximity of $\Omega_1$ and $\Omega_2$ we use the
best constant $\sigma=\sigma(H_1,H_2)$ in the inequality
\begin{equation}\label{Feb14a}
\mid(S_1-S_2)u\mid^2\leq\sigma ||u||^2,\;\;\; u\in\W^{1,2}(D),
\end{equation}
where $\mid\cdot\mid$ is the $L^2$-norm. In Sect. \ref{SectFeb15a}, we show that
the $\gamma$-convergence of domains is equivalent to their convergence in terms of
the distance (\ref{Feb14a}).

If $\sigma$ is sufficiently small, then problem (\ref{TTN2}) has exactly $J_m$
eigenvalues, say $\mu_1,\ldots,\mu_{J_m}$, in a neighborhood of $\lambda_m$; see
Proposition \ref{Pr1}. In order to formulate one of our main results we introduce
the following notation:

\noindent $T_2u=u-S_2u$, $\Psi=\Psi_\varphi\in H_2$ is the solution of the equation
\begin{equation}\label{TTN23feb}
(\Psi,w)=(\varphi,w)-\lambda_m\langle\varphi,w\rangle\;\;\;\mbox{for
all $w\in H_2$}
\end{equation}
and
\begin{equation}\label{Sept20cz}
\rho =\max_{\varphi\in X_m, ||\varphi||=1} \big(\mid
T\varphi\mid^2+\mid\Psi_\varphi\mid^2+\sigma||
\Psi_\varphi||^2)\big).
\end{equation}
Now we are in a position to formulate the following.

\begin{sats}\label{ThIntr1}  The asymptotic formula holds:
\begin{equation}\label{Sept20dz}
\mu_k^{-1}=\lambda_m^{-1}+\tau_k+O(\rho+|\tau_k|\sigma ),\;\; k=1,\ldots,J_m.
\end{equation}
Here $\tau=\tau_k$ is an eigenvalue of the problem
\begin{equation}\label{Eig19av1aaz}
\frac{1}{\lambda_m}\Big
((\Psi_\varphi,\Psi_\psi)-(T\varphi,T\psi)-(\Psi_\varphi,\psi)-(\varphi,\Psi_\psi)
\Big ) =\tau (S\varphi , S\psi ) \;\;\;\mbox{for all $\psi\in
X_m$,}
\end{equation}
where $\varphi\in X_m$; moreover, $\tau_1,\ldots,\tau_m$ in {\rm (\ref{Sept20dz})}
run through all eigenvalues of {\rm (\ref{Eig19av1aaz})} counted according to their
multiplicity.
\end{sats}

In the case when $\Omega_2\subset\Omega_1$, the function $\Psi_\varphi$ vanishes and
the eigenvalue problem (\ref{Eig19av1aaz}) takes the form:
\begin{equation}\label{Eig19av1aaz1}
-\frac{1}{\lambda_m}(T\varphi,T\psi)  =\tau (S\varphi , S\psi )
\;\;\;\mbox{for all $\psi\in X_m$.}
\end{equation}
If $\Omega_1\subset\Omega_2$, then $T\varphi=0$ and (\ref{Eig19av1aaz}) can be
written as follows:
\begin{equation}\label{Eig19av1aaz2}
\frac{1}{\lambda_m}(\Psi_\varphi,\Psi_\psi)  =\tau (S\varphi ,
S\psi ) \;\;\;\mbox{for all $\psi\in X_m$.}
\end{equation}

Similar theorems were proved in \cite{K1}\footnote{Two terms are lost in formula
(5), \cite{K1}, compare with (\ref{Sept20dz}) and (\ref{Eig19av1aaz}). However, if
$\Omega_2\subset\Omega_1$ or $\Omega_1\subset\Omega_1$, then formula (5), \cite{K1},
is true.} and \cite{KN}, but here, the main novelty is the use of the small
parameter $\sigma$ which makes the present theorem applicable to a substantially
larger class of perturbations. The importance of this parameter, that serves as a
distance in the set of all closed subspaces of $\W^{1,2}(D)$, lies in the fact that
this set is compact with respect to this distance (see Proposition \ref{PrFeb18a}).
As a result one obtain solutions of various optimization problems for functionals
that are continuous with respect to this distance. Another new point is that the
remainder in (\ref{Eig19av1aaz}) has the form, in which the small parameter $\sigma$
appears explicitly together with the quantities involved in the finite dimensional
spectral problem (\ref{Eig19av1aaz}).

Let us describe some applications of the asymptotic formula (\ref{Eig19av1aaz}). We
begin with the case when both $\Omega_1$ and $\Omega_2$ are Lipschitz domains with
the Lipschitz constant less than or equal to $C_*$ or, what is the same, uniformly
Lipschitz (see \cite{Ken} and \cite{He}, where the definitions of these notions are
given). We assume that
\begin{equation}\label{Intr19b}
\Omega_\varepsilon=\{x\in\Omega_1\,,\,\dist(x,\partial\Omega_1)
>\varepsilon\}\subset\Omega_2\subset\Omega_\varepsilon^+=\{x\in\Bbb
R^n\,,\,\dist(x,\Omega_1)<\varepsilon\},
\end{equation}
where $\varepsilon$ is a small positive parameter. It is supposed that the constant
$C_*$ is independent of $\varepsilon$. Then there exists a set $S$, subject to the
conditions
\[ \Omega_1\setminus\Omega_{2\varepsilon}\subset
S\subset\Omega_1 \quad \mbox{and} \quad |S|\leq C_1|\Omega_2\setminus\Omega_1| , \]
and such that the following inequality holds:
\begin{equation}\label{Intr19c}
\big |\mu_k^{-1}-\lambda_m^{-1}\big |\leq C_1\max_{\varphi\in
X_m,||\varphi||=1}\int_{(\Omega_1\setminus\Omega_2)\cup S} |\nabla \varphi|^2dx,\;\;
k=1,\ldots,J_m.
\end{equation}
Here $|S|$ is the surface measure of $S$ and the constants $C_1$ and $C_2$ do not
depend on $\varepsilon$. If $\Omega_2\subset\Omega_1$, then one has that $$ \big
|\mu_k^{-1}-\lambda_m^{-1}\big |\geq c\min_{\varphi\in
X_m,||\varphi||=1}\int_{(\Omega_1\setminus\Omega_2)} |\nabla \varphi|^2dx,\;\;
k=1,\ldots,J_m, $$ where the integral is the same as in the right-hand side of
(\ref{Intr19c}). If we assume that the eigenfunctions $\varphi_k$, $k=1,\ldots,J_m$,
belong to $W^{1,p}(\Omega_1)$ for some $p\in (2,\infty]$ then formula
(\ref{Intr19c}) implies that
\begin{equation}\label{Intr19d}
\big |\mu_k^{-1}-\lambda_m^{-1}\big |\leq
C|\Omega_1\bigtriangleup\Omega_2|^{1-\frac{2}{p}},
\end{equation}
where $\Omega_1\bigtriangleup\Omega_2$ denotes the symmetric difference of
$\Omega_1$ and $\Omega_2$. Estimate (\ref{Intr19d}) is equivalent to (1.8),
\cite{BL}, but the assumptions imposed on $\Omega_1$ and $\Omega_2$ are weaker
here; in particular, all egenfunctions are required to belong to
$W^{1,p}(\Omega_1)$ in \cite{BL}.

Let $\Omega_1$ be a bounded domain, $x_0\in\partial\Omega_1$ and
$\varepsilon>0$. We assume that
\begin{equation}\label{IntrJan26a}
\Omega_1\setminus
B_\varepsilon(x_0)\subset\Omega_2\subset\Omega_1\cup
B_\varepsilon(x_0)=\Omega_\varepsilon^+(x_0),
\end{equation}
where $B_\varepsilon(x_0)$ is the ball of the radius $\varepsilon$ centered at
$x_0$. Let also for all $u\in W^{1,2}(B_{q\varepsilon(x_0)})$ such that $u=0$ on
$B_{q\varepsilon}(x_0)\setminus\Omega_\varepsilon^+(x_0)$ the following inequality
\begin{equation}\label{Jan16a}
\int_{\Omega_\varepsilon^+(x_0)\cap
B_{q\varepsilon(x_0)}}|u|^2dx\leq C\varepsilon^2
\int_{\Omega_\varepsilon^+(x_0)\cap B_{q\varepsilon}(x_0)}|\nabla
u|^2dx
\end{equation}
hold with some $q>1$ and a constant $C$ independent of $\varepsilon$. Then it is
proved in Sect.\ref{SectLoca} that $$
 \big
|\mu_k^{-1}-\lambda_m^{-1}\big |\leq C\max_{\varphi\in
X_m,||\varphi||=1}\int_{\Omega_1\setminus
B_{q\varepsilon}(x_0)}|\nabla\varphi|^2dx,\;\; k=1,\ldots,J_m.
$$

In Sect.\ref{SectGla}, we consider perturbations satisfying (\ref{Intr19b}).
Assuming that for some $q>1$ inequality (\ref{Jan16a}) holds for all
$x_0\in\partial\Omega_1$, we prove the following estimate for the perturbed
eigenvalues: $$
 \big
|\mu_k^{-1}-\lambda_m^{-1}\big |\leq C\max_{\varphi\in
X_m,||\varphi||=1}\int_{\Omega_1\setminus
\Omega_{q\varepsilon}}|\nabla\varphi|^2dx,\;\; k=1,\ldots,J_m.
$$

Notice that the Sobolev space $\W^{1,2}(\Omega_k)$ can be considered as a subspace
of a similar Sobolev space in a larger domain. Thus, we present an abstract approach
for comparison of eigenvalues and eigenfunctions of operators acting in different
subspaces of a certain Hilbert space in Sections \ref{Sect1}--\ref{SectTh}. There,
in order to measure the proximity of two subspaces, we introduce a distance $\sigma$
as a norm of two projectors onto these subspaces and show that the convergence with
respect to this distance is equivalent to the Mosco convergence. In Sect.
\ref{Sect1}, we formulate a proposition about the closeness of eigenvalues and
eigenfunctions of two eigenvalue problems, see Proposition \ref{Pr1}. In the same
section, we present the main asymptotic theorem, see Theorem \ref{Th1}. Proposition
\ref{Pr1} and Theorem \ref{Th1} are proved in in Sections \ref{S4} and \ref{SectTh},
respectively. In Sect. \ref{SectApl}, we apply our asymptotic formula to eigenvalues
of the Dirichlet problem for a second-order differential operator. In particular, we
consider local and global perturbations of the boundary; the case of uniformly
Lipschitz boundaries is also treated. A class of domain perturbations, for which the
quantity $\sigma$ is small, is described for both cases; a stability estimate, which
evaluates eigenvalues by integrals of squares of the gradient of eigenfunctions, is
also presented.

\section{Perturbation of eigenvalues. Abstract
version.}\label{Sect1}

\subsection{Statement of the perturbation problem}\label{SectFeb15a}
Here we present an abstract approach for study of perturbation of
eigenvalues to the spectral problems from Introduction keeping the
same notations.

Let $H$ and ${\cal H}$ be Hilbert spaces with inner products
$(\cdot,\cdot )$ and $\langle\cdot,\cdot \rangle$ and with
corresponding norms  $||\cdot ||$ and $\mid \cdot \mid$
respectively. We suppose that $H$ is compactly  imbedded in ${\cal
H}$. This implies existence of $c_0>0$ such that
\begin{equation}\label{4.1a}
\mid u\mid\leq c_0 ||u||,\;\;\;\mbox{for $u\in H$}.
\end{equation}

Let $H_1$ and $H_2$ be two subspaces of $H$ of infinite dimension.
We introduce the operators $K_j :H_j\to H_j$ by $(K_ju,v)=\langle
u,v\rangle $, where $u,\,v\in H_j$, $j=1,2$. One can check that the
operators $K_1$ and $K_2$ are self-adjoint, positive definite and
compact. Let ${\cal H}_j$, $j=1,2$, be the closure of $H_j$ in the
space ${\cal H}$. From the definition of $K_j$, $j=1,2,$ it follows
that the operator can be extended to ${\cal H}_j$ and
$$
||K_j||_{{\cal H}_J\to H_j}\leq c_0,
$$
where $c_0$ is the constant in (\ref{4.1a}).

We consider two spectral problems
\begin{equation}\label{1.1}
K_1\varphi=\lambda^{-1}\varphi,\;\;\varphi\in H_1,
\end{equation}
and
\begin{equation}\label{1.2}
K_2U=\mu^{-1}U,\;\;U\in H_2,
\end{equation}
We denote by $\lambda^{-1}_j$, $j=1,\ldots$,  eigenvalues of
$K_1$, numerated according to $0<\lambda_1<\lambda_2<\cdots$, and
by $X_j$ corresponding eigenspaces. We put $J_j=\dim X_j$.   Our
goal is to study  eigenvalues of (\ref{1.2}) located in a
neighborhood of $\lambda_m$ for certain fixed $m$.\footnote{We
note that the spectral problem (\ref{1.1}) and (\ref{TTN1}), and
also (\ref{1.2}) and (\ref{TTN2}) have the same eigenvalues and
eigenvectors.}

 We denote by $S_j$, $j=1,2,$ the orthogonal projector in $H$ with the image $H_j$.
 We will measure the proximity
between  $H_1$ and $H_2$ by the constant $\sigma=\sigma(H_1,H_2)$
in the inequality
\begin{equation}\label{1.1azz}
\mid (S_1-S_2)u \mid^2 \leq \sigma ||u||^2,\;\; u\in H.
\end{equation}
The quantity $\sigma(H_1,H_2)$ is a distance in  ${\cal S}(H)$-the
space of all closed subspaces of $H$. In the next proposition we
prove that the Mosco convergence of subspaces (see Sect. 2.3.3 in
\cite{He} or Sect.4.5 in \cite{B}) is equivalent to convergence of
subspaces with respect to the distance $\sigma$.

\begin{prop}\label{PrFen16a} Let $H_j$, $j=1,\ldots$, and $H_*$ be subspaces in $H$
and let $S_j$ and $S_*$ be corresponding orthogonal projectors.
the following assertions are equivalent:

\noindent {\rm (i)} for all $u\in H$, $S_ju\to S_*u$ as
$j\to\infty$;

\noindent {\rm (ii)} $\sigma (H_j,H_*)\to 0$ as $j\to\infty$.
\end{prop}
\begin{proof} We denote by $0<s_1\leq s_2\leq\cdots$,  eigenvalues of
 problem (\ref{1.1}), when $H_1=H$. Here, we numerate them
 accounting  their multiplicities. Due to the compactness
 of the imbedding $H\subset {\cal H}$, $s_k\to\infty$ as
 $k\to\infty$. We denote by $v_1,\,v_2,\ldots$, corresponding
 eigenvectors which form an orthogonal basis in $H$. We normalize
 them according to $||v_k||=1$. It is clear that
 $$
(v_k,w)=s_k\langle v_k,w\rangle\;\;\;\mbox{for all $w\in H$.}
 $$
Therefore, the vectors $w_k=\sqrt{s_k}v_k$, $k=1,2,\ldots$, form the
orthogonal basis in ${\cal H}$ subject to $\mid w_k\mid=1$.

Let us show that (i) implies (ii). We have
$$
(S_j-S_*)u=\sum_{k=1}^N s_k^{-1/2}
(u,(S_j-S_*)v_k)w_k+\sum_{k=N+1}^\infty s_k^{-1/2}
((S_j-S_*)u,v_k)w_k.
$$
Therefore,
$$
\mid(S_j-S_*)u\mid^2 \leq\Big (\sum_{k=1}^N
s_k^{-1}||(S_j-S_*)v_k||^2+s_{N+1}^{-1}\Big )||u||^2.
$$
this implies
$$
\sigma (H_j,H_*)\leq \inf_N\Big (\sum_{k=1}^N
s_k^{-1}||(S_j-S_*)v_k||^2+s_{N+1}^{-1}\Big ).
$$
The right hand side in the last inequality tends to zero when
$j\to\infty$ because of strong convergence of $S_j$ to $S_*$.

Let us prove the implication (ii) $\Rightarrow$ (i). For $u\in H$,
we have
$$
((S_j-S_*)u,v_k)=s_k\langle (S_j-S_*)u,v_k\rangle \rightarrow
0\;\;\;\mbox{as $j\to\infty$}
$$
for $k=1,2,\ldots$ This implies that $((S_j-S_*)u,v)\to 0$ as
$j\to\infty$ for all $v\in H$. Therefore,
\begin{eqnarray*}
&&((S_j-S_*)u,(S_j-S_*)u)=(S_ju,u)-(S_*u,S_ju)-(S_ju,S_*u)+(S_*h,S_*u)\\
&&\rightarrow
(S_*u,u)-(S_*u,S_*u)-(S_*u,S_*u)+(S_*h,S_*u)=0\;\;\;\mbox{as
$j\to\infty$}.
\end{eqnarray*}
The proof is complete.
\end{proof}

Using equivalence of $\gamma$ convergence of domains and strong
convergence of corresponding operators (see Theorem 2.3.10 in
\cite{H}), we derive from the previous assertion the following
\begin{kor}
Let $\Omega_j$, $j=1,\ldots,$ and $\Omega_*$ be domains belonging
to a bounded domain $D$. If $\Omega_j$ $\gamma$-converges to
$\Omega_*$ then $\sigma (H_j,H_*)\to 0$ as $j\to\infty$, where
$H_j=\W^{1,2}(\Omega_j)$ and $H_*=\W^{1,2}(\Omega_*)$.
\end{kor}

\begin{prop}\label{PrFeb18a} The metric space ${\cal S}(H)$
with the distance $\sigma$ is compact.
\end{prop}
\begin{proof} Let $S_k$, $k=1,2,\ldots$, be a family of orthogonal
projectors. Let us show that one can choice a convergent
subsequence. Let $v_j$ be the same vectors as in the proof of
Proposition \ref{PrFen16a}. We can choose a subsequence such that
\begin{equation}\label{Feb16a}
(S_kv_i,v_j)\to \alpha_{ij}\;\;\;\mbox{as $k\to\infty$}
\end{equation}
for all $i,j\geq 1$. We used the same index for the subsequence in
(\ref{Feb16a}). Since $||S_k||_{H\to H}=1$, we derive from
(\ref{Feb16a}) that
$$
(S_ku,v)\to \alpha(u,v)\;\;\;\mbox{as $k\to\infty$}
$$
for all $u,v\in H$, where $\alpha(u,v)$ is a certain number. One
can check that the form $\alpha$ is linear with respect to the
first argument and anti-linear with respect to the second one.
Moreover, the form $\alpha$ is bounded. Therefore, $\alpha(u,v)=
(S_*u,v)$, where $S_*$ is an orthogonal projector. Reasoning as in
the proof of Proposition \ref{PrFen16a}(ii), we conclude that
$S_ku\to S_*u$ as $k\to\infty$.
\end{proof}

In what follows we shall use that the orthogonal projectors $S_1$
and $S_2$ posses the following symmetry property:
$$
(S_2v,w)=(v,S_1w)\;\;\;\mbox{for $v\in H_1$ and $w\in H_2$.}
$$

\subsection{Formulation of results}

In what follows we denote by  $P_m$ the orthogonal projector in
$H$ with the image $SX_m$.

\begin{prop}\label{Pr1}
There exists  positive $\sigma_0$,  $c$ and $C$ depending on
$\lambda_1,\ldots,\lambda_{m+1}$ and $c_0$ such that for
$\sigma\leq \sigma_0$ the following assertions are valid:

{\rm (i)} The operator $K_2$ has exactly $J_m$ eigenvalue in
$(1/\lambda_{m+1}+c\sigma^{1/2},1/\lambda_{m-1}-c\sigma^{1/2})$
and all of them are located in
$(1/\lambda_{m}-c\sigma^{1/2},1/\lambda_{m}+c\sigma^{1/2})$.

{\rm (ii)} If $\mu^{-1}$ is an eigenvalue of {\rm (\ref{1.2})}
from the interval
$(1/\lambda_{m}-c\sigma^{1/2},1/\lambda_{m}+c\sigma^{1/2})$ and
$U$ is a corresponding eigenvector then
\begin{equation}\label{0.1}
||U-P_mU||\leq C\sigma^{1/2}||U||.
\end{equation}
\end{prop}

We denote by $\mu_k^{-1}$, $k=1,\ldots,J_m$, the eigenvalues of
the spectral problem {\rm (\ref{1.2})} located in the interval
$(\lambda_m^{-1}-c\sigma,\lambda_m^{-1}+c\sigma)$, where $c$ is
the same positive constant as in {\rm Proposition \ref{Pr1}}. In
order to formulate the main result of this paper let us introduce
some more objects. We put $T_2u=u-S_2u$ and define the vector
$\Psi=\Psi_\varphi\in H_2$ as the solution of equation
(\ref{TTN23feb}).  Let also
\begin{equation}\label{Sept20c}
\rho =\max_{\varphi\in X_m, ||\varphi||=1} \big(\sigma ||
\Psi_\varphi||^2+\mid T\varphi\mid^2+\mid \Psi_\varphi\mid^2\big).
\end{equation}

\begin{sats}\label{Th1}  The
following asymptotic formula holds:
\begin{equation}\label{Sept20d}
\mu_k^{-1}=\lambda_m^{-1}+\tau_k+O(\rho+|\tau_k|\sigma ),\;\;
k=1,\ldots,J_m,
\end{equation}
where $\tau=\tau_k$ is an eigenvalue of the problem
\begin{equation}\label{Eig19av1aa}
\frac{1}{\lambda_m}\Big
((\Psi_\varphi,\Psi_\psi)-(T\varphi,T\psi)-(\Psi_\varphi,\psi)-(\varphi,\Psi_\psi)
\Big ) =\tau (S\varphi , S\psi ) \;\;\;\mbox{for all $\psi\in
X_m$,}
\end{equation}
where $\varphi\in X_m$. Moreover, $\tau_1,\ldots,\tau_m$ in {\rm
(\ref{Sept20d})} run through all eigenvalues of {\rm
(\ref{Eig19av1aa})} counting their multiplicities.
\end{sats}

In Sect. \ref{Sect4} we prove that $||\varphi||^2
-||S\varphi||^2\leq 1-C\sigma^{1/2}$ for $\varphi\in X_m$. This
fact and Theorem \ref{Th1} lead to the following corollaries.

\begin{kor}\label{Kor1} If $H_2\subset H_1$ then $\Psi_\varphi=0$ for all
$\varphi\in X_m$. Therefore, from {\rm (\ref{Eig19av1aa})} it
follows that $|\tau_k|\leq C\rho$ and hence {\rm (\ref{Sept20d})}
implies
$$
c\min_{\varphi\in X_m, ||\varphi||=1} ||T\varphi||^2\leq\big
|\mu_k^{-1}-\lambda_m^{-1}\big |\leq C\max_{\varphi\in X_m,
||\varphi||=1} ||T\varphi||^2,
$$
where $c$ and $C$ are positive constants.
\end{kor}
\begin{kor}\label{Kor2} Let $S_0$ be the orthogonal projector onto
$H_1\cap H_2$ and let $T_0u=u-S_0u$. Then
$(\Psi_\varphi,\psi)=(\Psi_\varphi,T_0\psi)$ and therefore from
{\rm (\ref{Eig19av1aa})} it follows $|\tau_k|\leq C\rho_0$, where
$$
\rho_0=\max_{\varphi\in X_m, ||\varphi||=1} (|| T_0\varphi||^2+||
\Psi_\varphi||^2).
$$
Consequently,
$$
\big |\mu_k^{-1}-\lambda_m^{-1}\big |\leq C\rho_0
$$
for arbitrary  $H_1$ and $H_2$ subject to {\rm (\ref{1.1azz})}
with small constant $\sigma$.
\end{kor}

\begin{Rem} In {\rm \cite{K1}} and {\rm \cite{KN}} another quantity is used to
measure proximity of two subspaces. I is defined as a constant
$\sigma_*=\sigma_*(H_1,H_2)$ in the inequality
\begin{equation}\label{1.1a}
\mid u\mid^2\leq \sigma_* ||u||^2
\end{equation}
 for all elements $u$ from $u\in H_1+H_2$
orthogonal to $H_1\cap H_2$. Since
$$
\mid(S_1-S_2)u\mid\leq \sqrt{\sigma_*}||(S_1-S_2)u||\leq
2\sqrt{\sigma_*}||u||,
$$
the constant $\sigma$ in {\rm (\ref{1.1azz})} is subject to
\begin{equation}\label{Feb10a}
\sigma\leq 4\sigma_*.
\end{equation}
One can check that $T\varphi$ and $\Psi_\varphi$ belong to
$H_1+H_2$ and both of them are orthogonal to $H_1\cap H_2$ for
$\varphi\in X_m$. Therefore, $\mid T\varphi\mid\leq\sigma_*
||T\varphi||^2$ and $\mid \Psi_\varphi\mid\leq\sigma_*
||\Psi_\varphi||^2$ due to {\rm (\ref{1.1a})}. So, under
assumption (\ref{1.1a}), we have
\begin{equation}\label{Feb10b}
\rho\leq c\sigma_*\rho_*,\;\;\; \rho_*=\max_{\varphi\in X_m,
||\varphi||=1}(||T\varphi||^2+||\Psi_\varphi||^2).
\end{equation}
Using  {\rm (\ref{Feb10b})}, we see that formula {\rm
(\ref{Sept20d})} in {\rm Theorem \ref{Th1}} implies the following
asymptotic formula
\begin{equation}\label{Sept20dzz}
\mu_k^{-1}=\lambda_m^{-1}+\tau_k+O(\sigma_*\rho_*+|\tau_k|\sigma_*
),\;\; k=1,\ldots,J_m,
\end{equation}
where $\tau=\tau_k$ is an eigenvalue of the problem {\rm
(\ref{Eig19av1aaz})}.  Moreover, $\tau_1,\ldots,\tau_m$ in {\rm
(\ref{Sept20dzz})} run through all eigenvalues of {\rm
(\ref{Eig19av1aaz})} counting their multiplicities.
\end{Rem}

\section{Proof of Proposition \protect\ref{Pr1}}\label{S4}

In what follows, we put ${\cal X}_m=X_1+\cdot +X_m$ and by ${\cal
Y}_m$ we denote the orthogonal  complement of $S{\cal X}_m$ in
$H_2$.

An important role will be played by the operator
\begin{equation}\label{Sept28a}
B=K_2S_2-S_2K_1.
\end{equation}
In Sect \ref{Sect4}, we will show in particular,  that the norm of
$B:H_1\to H_2$ is estimated by a constant times $\sigma^{1/2}$.

By $c$, $C$,... we denote various constants depending on
$\lambda_1,\cdots,\lambda_{m+1}$ and $c_0$.

\subsection{Some inequalities}\label{Sect4}

In this section we prove some important estimates, which follow
from (\ref{1.1azz}) and which will be used in the proofs of
Proposition \ref{Pr1} and Theorem \ref{Th1}.

{\em Estimates for $S_2$}. Let us start with the following
inequality
\begin{equation}\label{4.1}
\Big (1-\Lambda_m\sigma^{1/2}\Big
)||\varphi||^2\leq||S\varphi||^2\leq ||\varphi||^2\;\;\;\mbox{for
$\varphi\in {\cal X}_m$},
\end{equation}
where $\Lambda_m=\Big (\sum_{k=1}^m\lambda_k\Big )^{1/2}$. Since
$S_2$ is an orthogonal projector, the right inequality is obvious.
Let us prove the left one. We represent $\varphi$ as
$\varphi_0+\varphi_1$, where $\varphi_0=S_2\varphi$ and
$\varphi_1=(I-S_2)\varphi$. Since $S_2\varphi_0=\varphi_0$ and
$S_2\varphi_1=0$, we have $||S\varphi||^2=||\varphi_0||^2$.
 Using that
$S_1\varphi=\varphi$, we conclude that
$\varphi_1=(S_1-S_2)\varphi$ and from (\ref{1.1azz}) it follows
that
\begin{equation}\label{4.2}
\mid\varphi_1\mid^2=\mid(S_1-S_2)\varphi\mid^2\leq\sigma||\varphi||^2.
\end{equation}
Representing $\varphi$ as $\varphi=\zeta_1+\ldots+\zeta_m$, where
$\zeta_k\in X_k$, and noting that
$(\varphi_1,\varphi_1)=(\varphi,\varphi_1)$, we get
\begin{eqnarray*}
&&(\varphi_1,\varphi_1)=\sum_{k=1}^m(\zeta_k,\varphi_1)
=\sum_{k=1}^m\lambda_k\langle\zeta_k,\varphi_1\rangle\leq
\sum_{k=1}^m\lambda_k\mid\zeta_k\mid\,\mid\varphi_1\mid\\
&&\leq
\sum_{k=1}^m\lambda_k^{1/2}||\zeta_k||\,\mid\varphi_1\mid\leq
\Lambda_m||\varphi||\mid\varphi_1\mid ,
\end{eqnarray*}
where we used that $\zeta_k$ is an eigenvector of $K_1$
corresponding to the eigenvalue $\lambda_k^{-1}$. Using the last
inequality together with (\ref{4.2}), we get
$$
(\varphi_1,\varphi_1) \leq\Lambda_m\sigma^{1/2}(\varphi,\varphi).
$$
Therefore,
$$
||S\varphi||^2= ||\varphi_0||^2=||\varphi||^2-||\varphi_1||^2\geq
(1-\Lambda_m\sigma^{1/2})||\varphi||^2,
$$
which implies the left inequality in (\ref{4.1}).

From (\ref{4.1}) one can derive  the estimate
\begin{equation}\label{4.2a}
|(S\varphi,S\psi)-(\varphi,\psi)|\leq 3\Lambda_m
\sigma^{1/2}(||\varphi||^2+||\psi||^2)\;\;\mbox{for
$\varphi,\psi\in{\cal X}_m$.}
\end{equation}
Indeed, introduce the form
$b(u,v)=(S\varphi,S\psi)-(\varphi,\psi)$. By (\ref{4.2}),
\begin{equation}\label{4.2b}
|b(u,u)|\leq \varepsilon ||u||^2,
\end{equation}
where $\varepsilon=\Lambda_m\sigma^{1/2}$. Using that
$$
b(\varphi+\psi,\varphi+\psi)=b(\varphi,\varphi)+b(\psi,\psi)
+b(\varphi,\psi)+b(\psi,\varphi)
$$
and applying (\ref{4.2b}) for estimating quadratic terms here we
obtain
$$
|b(\varphi,\psi)+b(\psi,\varphi)|\leq
3\varepsilon(||\varphi||^2+||\psi||^2).
$$
Similar arguments applied to $b(\varphi+i\psi,\varphi+i\psi)$ give
the estimate
$$
|b(\varphi,\psi)-b(\psi,\varphi)|\leq
3\varepsilon(||\varphi||^2+||\psi||^2),
$$
which together with the previous one leads to (\ref{4.2a}).

{\em An estimate for the operator $B$}. Let us prove that
\begin{equation}\label{4.3}
||Bv||\leq 2c_0\sigma^{1/2}||v||\;\;\;\mbox{for $v\in H_1$}.
\end{equation}
For $w\in H_2$ we have
\begin{equation}\label{4.4}
(Bv,w)=(K_2S_2v,w)-(S_2K_1v,w)=\langle S_2v,w\rangle-\langle
v,S_1w\rangle.
\end{equation}
We write $v$ and $w$ as $v=S_2v+(I-S_2)v=v_0+v_1$ and
$w=S_1w+(I-S_1)w=w_0+w_1$. Then (\ref{4.4})  implies
\begin{equation}\label{4.5}
(Bv,w)=\langle v_0,w_1\rangle-\langle v_1,w_0\rangle,
\end{equation}
where we have used the equalities $S_2v_1=S_1w_1=0$. Since
$v_1=(S_1-S_2)v$ and $w_1=(S_2-S_1)w$, then by using estimate
(\ref{1.1azz}) for function containing index $1$, we get
$$
|(Bv,w)|\leq\sigma^{1/2}\big (\mid v_0\mid\,||w||+||v||\,\mid
w_0\mid\big).
$$
Applying here estimate (\ref{4.1a}), we get
$$
|(Bv,w)|\leq 2c_0\sigma^{1/2}||v||\,||w||,
$$
which implies (\ref{4.3}).

{\em An inequality for $K_1$ and $K_2$}. Finally, let us show that
\begin{equation}\label{4.6}
(K_2w,w)\leq
(K_1S_1w,S_1w)+(2c_0\sqrt{\sigma}+\sigma)||w||^2\;\;\;\mbox{for
$w\in H_2$},
\end{equation}
or what is equivalent, due to the definition of $K_1$ and $K_2$,
\begin{equation}\label{4.6a}
\mid w\mid^2\leq \mid
S_1w\mid^2+(2c_0\sqrt{\sigma}+\sigma)||w||^2.
\end{equation}
We write $w=S_1w+(I-S_1)w=w_0+w_1$. Since $S_1w_1=0$ and
$S_1w_0=w_0$, relation (\ref{4.6a}) takes the form
\begin{equation*}
\mid w_0+w_1\mid^2\leq\mid
w_0\mid^2+(2c_0\sqrt{\sigma}+\sigma)||w||^2.
\end{equation*}
or
\begin{equation}\label{4.6b}
\mid w_1\mid^2+2\langle w_0,w_1\rangle\leq
(2c_0\sqrt{\sigma}+\sigma)||w||^2.
\end{equation}
Using that $w_1=(S_2-S_1)w$ and applying (\ref{1.1azz}), we
estimate the left-hand side of (\ref{4.6b}) by
$$
\sigma||w||^2+2\sigma^{1/2}\mid w_0\mid\,||w||.
$$
According to (\ref{4.1a}), $\mid w_0\mid\leq c_0||w_0||\leq
c_0||w||$, which implies (\ref{4.6b}) and hence (\ref{4.6}).

\subsection{Proof of Proposition \protect\ref{Pr1}(i)}

1) Let $u\in {\cal Y}_m$. Then $S_1u$ is orthogonal to
$X_1+\cdots+X_m$ and by (\ref{4.6})
\begin{eqnarray}\label{4.6c}
&&(K_2u,u)\leq (K_1S_1u,S_1u)+c\sigma^{1/2}||u||^2\leq
\frac{1}{\lambda_{m+1}}(S_1u,S_1u)+c\sigma^{1/2}||u||^2\nonumber\\
&&\leq \Big (\frac{1}{\lambda_{m+1}}+c\sigma^{1/2}\Big )||u||^2.
\end{eqnarray}
 From this inequality it follows that there
are $\leq J_1+\cdots+J_m$ eigenvalues of $K_2$ counted together
with their multiplicity in the interval
$(1/\lambda_{m+1}+c\sigma^{1/2},\infty)$.

2) Let $u=S_2\varphi$, $\varphi\in {\cal X}_{m}$. Then
$$
(K_2u,u)=(S_2K_1\varphi,S_2\varphi)+(B\varphi,u)
$$
and using (\ref{4.3}), we get
$$
(K_2u,u)\geq(S_2K_1\varphi,S_2\varphi)-C\sigma^{1/2}||\varphi||\,||u||.
$$
Applying (\ref{4.2a}) to the first term in the left-hand side, we
obtain
$$
(K_2u,u)\geq(K_1\varphi,\varphi)-C_1\sigma^{1/2}(||K_1\varphi||^2+||\varphi||^2)
-C\sigma^{1/2}||\varphi||\,||u||,
$$
which leads to
$$
(K_2u,u)\geq(K_1\varphi,\varphi)-C_2\sigma^{1/2}(||\varphi||^2\geq
\Big(\frac{1}{\lambda_{m-1}}-C_2\sigma^{1/2}\Big )||\varphi||^2.
$$
Applying (\ref{4.1}) for estimating the second term in the
right-hand side in the last inequality, we arrive at
\begin{equation*}
(K_2u,u)\geq \Big(\frac{1}{\lambda_{m}}-C_2\sigma^{1/2}\Big
)||u||^2.
\end{equation*}
This implies that there are $\geq J_1+\cdots +J_{m}$ eigenvalues
of $K_2$ in the interval $(1/\lambda_{m}-C_2\sigma^{1/2},\infty)$.

Combining 1) and 2) we conclude that there are exactly
$J_1+\cdots+J_m$ eigenvalues of $K_2$ in the interval
$(\gamma_1,\infty)$ for $\gamma_1\in
(1/\lambda_{m+1}+C_1\sigma^{1/2},1/\lambda_{m}-C_2\sigma^{1/2})$ .

Applying 1) and 2) with $m$ replaced by $m-1$ we obtain that there
are exactly $J_1+\cdots+J_{m-1}$ eigenvalues of $K_2$ in the
interval $(\gamma_2,\infty)$ for $\gamma_2\in
(1/\lambda_{m}+C_1\sigma^{1/2},1/\lambda_{m-1}-C_2\sigma^{1/2})$ .
Therefore there exists  positive constants $c$ and $\sigma_0$
depending on $\lambda_1,\ldots,\lambda_{m+1}$ and $c_0$ such that
the intervals
$(1/\lambda_{m}-c\sigma^{1/2},1/\lambda_{m}+c\sigma^{1/2})$ and
$(1/\lambda_{m+1}+c\sigma^{1/2},1/\lambda_{m-1}-c\sigma^{1/2})$
contains exactly $J_m$ eigenvalues of $K_2$ for
$\sigma\leq\sigma_0$. The proof of Proposition \ref{Pr1}(i) is
complete.

\subsection{Proof of Proposition \protect\ref{Pr1}(ii)}

First, let us consider the equation
\begin{equation}\label{7.1}
Q_mK_2w-\mu w=f,
\end{equation}
where $f,w\in Q_mH_2$, $Q_m=I-P_m$ and
\begin{equation}\label{Sept28b}
|\mu^{-1}-\lambda_m^{-1}|\leq c\sigma^{1/2},
\end{equation}
where $c$ is the same constant as in (i). Our first goal is to
prove the estimate
\begin{equation}\label{7.8}
||w||\leq c_1||f||
\end{equation}
for solutions of equation (\ref{7.1}). Here the constant $c_1$
depends on $\lambda_1,\cdots,\lambda_{m+1}$ and $c_0$.

We represent $S_2{\cal X}_m$ as $S_2X_m+ Y_m$ where $Y_m$ is the
orthogonal complement to $S_2X_m$ in $S_2{\cal X}_m$. We introduce
the orthogonal projectors $R_m$ onto $Y_m$ and ${\cal T}_m$ onto
${\cal Y}_m$. Then $Q_m=R_m+ {\cal T}_m$. We write $w=w_0+w_1$ and
$f=f_0+f_1$, where $w_0,f_0\in Y_m$ and $w_1,f_1\in {\cal Y}_m$.
Equation (\ref{7.1}) can be written as the following system of
equations
\begin{equation}\label{7.1a}
R_mK_2(w_0+w_1)-\mu w_0=f_0
\end{equation}
and
\begin{equation}\label{7.2}
{\cal T}_mK_2w_1-\mu w_1=f_1-{\cal T}_mK_2w_0.
\end{equation}
From the second equation we get
$$
(K_2w_1,w_1)-\mu (w_1,w_1)=(f_1-{\cal T}_mK_2w_0,w_1).
$$
Using estimate (\ref{4.6c}), we obtain
$$
\mu(w_1,w_1)-\Big (\frac{1}{\lambda_{m+1}}+c\sigma^{1/2}\Big
)||w_1||^2\leq ({\cal T}_mK_2w_0,w_1)-(f_1,w_1),
$$
which implies that the operator $\mu^{-1}-Q_mK_2$ is positive
definite on ${\cal Y}_m$, equation (\ref{7.2}) is uniquely
solvable and its solution satisfies
\begin{equation}\label{7.3}
\Big (\mu-\frac{1}{\lambda_{m+1}}-c\sigma^{1/2}\Big )||w_1||_2\leq
||{\cal T}_mK_2w_0||+||f_1||.
\end{equation}
Furthermore, representing $w_0$ as $S\varphi_0$, $\varphi_0\in
{\cal X}_m$ we have
$$
{\cal T}_mK_2w_0={\cal T}_mK_2S_2\varphi_0={\cal
T}_m(S_2K_1\varphi_0+B\varphi_0)={\cal T}_mB\varphi_0.
$$
Therefore, by (\ref{4.3}) and (\ref{4.1})
\begin{equation}\label{7.3d}
||{\cal T}_mK_2w_0||\leq c\sigma^{1/2}||\varphi_0||\leq
c_1\sigma^{1/2}||w_0||.
\end{equation}
Combining this estimate with (\ref{7.3}), we get
\begin{equation}\label{7.3e}
||w_1||\leq C(||f_1||+\sigma^{1/2}||w_0||).
\end{equation}
We represent $\varphi_0$ as $\varphi'+\varphi_m$, where
$\varphi_m\in X_m$ and $\varphi'\in X_1+\cdots +X_{m-1}$. Since
$(S_2\varphi_0,S_2\varphi_m)=0$, using (\ref{4.2a}) we get
\begin{equation}\label{7.4}
(\varphi_0,\varphi_m)=||\varphi_m||^2\leq 3\lambda_m^{1/2}
\sigma^{1/2}||\varphi_0||^2.
\end{equation}
Therefore,
\begin{equation*}\label{7.5}
(K_2w_0,w_0)=(K_2S_2\varphi_0,S_2\varphi_0)=(S_2K_1\varphi_0,S_2\varphi_0)
+(B\varphi_0,S_2\varphi_0).
\end{equation*}
Using (\ref{4.2a}) and (\ref{4.3}) for estimating the first and
second terms in the right-hand side of the last relation
respectively we get
\begin{equation}\label{7.6}
(K_2w_0,w_0)\geq
(K_1\varphi_0,\varphi_0)-c\sigma^{1/2}||\varphi_0||^2,
\end{equation}
where $c$ depends on the eigenvalues $\lambda_1,\lambda_m$ and the
constant $c_0$ in (\ref{4.1a}). Since
$$
(K_1\varphi_0,\varphi_0)=(K_1\varphi_m,\varphi_m)
+(K_1\varphi',\varphi')=\frac{1}{\lambda_m}||\varphi_m,||^2+(K_1\varphi',\varphi'),
$$
by using (\ref{7.4}), we get
$$
(K_1\varphi_0,\varphi_0)\geq \Big
(\frac{1}{\lambda_{m-1}}-c\sigma^{1/2}\Big )||\varphi_0||^2\geq
(\frac{1}{\lambda_{m-1}}-c_1\sigma^{1/2}\Big )|w_0||^2.
$$
Here we used also (\ref{4.1}) in order to obtain the last
inequality. Applying the last estimate to the first term in the
right-hand side in (\ref{7.6}), we obtain
\begin{equation}\label{7.7}
(K_2w_0,w_0)\geq(\frac{1}{\lambda_{m-1}}-c_2\sigma^{1/2}\Big
)|w_0||^2.
\end{equation}
Using (\ref{7.7}) and (\ref{7.3d}), we conclude that equation
(\ref{7.1a}) is solvable with respect to $w_0$ and $||w_0||\leq
c||f||$. Similar estimate for $w_1$ follows from (\ref{7.3e}).
Conclusively, equation (\ref{7.1}) is uniquely solvable and for
its solution $w\in Q_mH_2$  estimate (\ref{7.8}) holds.

\bigskip Let $\mu^{-1}$ be an eigenvalue of problem (\ref{1.2})
satisfying (\ref{Sept28b}) and let $U$ be a corresponding
eigenvector. We represent it as $U=S_2\varphi_m+w$ where
$\varphi_m\in X_m$ and $w\in w\in Q_mH_2$. Then
$$
K_2(S_2\varphi_m+w)=\frac{1}{\mu}(S_2\varphi_m+w).
$$
We write this relation as
$$
K_2w-\frac{1}{\mu}w+S_2K_1\varphi_m+B\varphi_m-\frac{1}{\mu}S_2\varphi_m=0,
$$
or, equivalently
\begin{equation}\label{7.9}
K_2w-\frac{1}{\mu}w=-B\varphi_m+\Big
(\frac{1}{\mu}-\frac{1}{\lambda_m}\Big )S_2\varphi_m.
\end{equation}
We denote  the left-hand side  by $f$. Then by (\ref{4.3}) and
assumption on $\mu$, we have
$$
||f||\leq C\sigma^{1/2}||\varphi_m||.
$$
Applying operator ${\cal T}_m$ to both sides in (\ref{7.9}) and
using (\ref{7.8}) together with the last estimate of $f$ we obtain
\begin{equation}\label{7.10}
||w||\leq C\sigma^{1/2}||\varphi_m||\leq C_1\sigma^{1/2}||w||.
\end{equation}

\subsection{Corollary of Proposition \protect\ref{Pr1}}

Let $Z_m$ be the space of eigenvectors of the problem (\ref{1.2})
corresponding to eigenvalues located in the interval
$(1/\lambda_{m}-c\sigma^{1/2},1/\lambda_{m}+c\sigma^{1/2})$, see
Proposition \ref{Pr1}(i). According to the same proposition $\dim
Z_m=J_m$. Let us show that
\begin{equation}\label{Okt13a}
|(U,V)-(P_mU,P_mV)|\leq C\sigma (||U||^2+||V||^2)\;\;\;\mbox{for
all $U,V\in Z_m$.}
\end{equation}
First, let us check that
\begin{equation}\label{Okt13b}
(U,U)-(P_mU,P_mU)\leq C\sigma ||U||^2\;\;\;\mbox{for all $U\in
Z_m$.}
\end{equation}
Indeed, introduce an orthonormal basis $U_1,\ldots,U_{J_m}$ in
$Z_m$ consisting of eigenvectors of problem (\ref{1.2}). If $U\in
Z_m$ we represent it as $U=a_1U_1+\cdots+a_{J_m}U_{J_m}$. Using
(\ref{0.1}), we get
$$
||U-P_mU||\leq\sum_{j=1}^{J_m}|a_j|\,||U_j-P_mU_j||\leq
c\sigma^{1/2}\sum_{j=1}^{J_m}|a_j|\leq c_1\sigma^{1/2}||U||,
$$
which implies (\ref{Okt13b}) since
$||U-P_mU||^2=(U,U)-(P_mU,P_mU)$.

In order to prove (\ref{Okt13a}) we introduce the quasi-linear
form $b(U,V)=(U,V)-(P_mU,P_mV)$. Since $b(U,U)\geq 0$, we have
$$
|b(U,V)|\leq b(U,U)^{1/2}b(V,V)^{1/2}\leq
\frac{1}{2}(b(U,U)+b(V,V)),
$$
which together with (\ref{Okt13b}) implies (\ref{Okt13a}).

\section{Proof of Theorem \protect\ref{Th1}}\label{SectTh}

\subsection{A finite dimensional reduction}\label{S2}

 We represent the function $U\in H_2$ in (\ref{TTN2})
as $U=S_2\varphi+w$, where $\varphi\in X_m$ and $w\in Q_mH_2$.
Then (\ref{TTN2}) takes the form
$$
(\mu^{-1}- K_2)(S_2\varphi+w)=0.
$$
By using the operator $B$, we can write the last relation as
$$
\Big (\frac{1}{\mu}-\frac{1}{\lambda_m}\Big
)S_2\varphi-B\varphi+(\mu^{-1}-K_2)w=0
$$
Applying operators $P_m$ and $Q_m$ we get
\begin{equation}\label{2.1}
\Big (\frac{1}{\mu}-\frac{1}{\lambda_m}\Big
)S_2\varphi-P_mB\varphi-P_mK_2w=0
\end{equation}
and
\begin{equation}\label{2.1m}
Q_m(\mu^{-1}-K_2)w=Q_mB\varphi.
\end{equation}
We assume that $\mu$ satisfies (\ref{Sept28b}). Then the last
equation coincides with (\ref{7.1}) if we take there
$f=-Q_mB\varphi$. Therefore equation (\ref{2.1m}) is uniquely
solvable and
\begin{equation}\label{LUD3a}
||(\mu^{-1}-Q_mK_2Q_m)^{-1}||_{Q_mH_2\to Q_mH_2}\leq c_1,
\end{equation}
where   $c_1$ is the  constant from (\ref{7.8}). Inserting
$w=(\mu^{-1}-Q_mK_2Q_m)^{-1}Q_mB\varphi$ in (\ref{2.1}), we obtain
\begin{eqnarray}\label{2.2}
&&\Big (\frac{1}{\mu}-\frac{1}{\lambda_m}\Big
)S_2\varphi-P_mB\varphi-R(\lambda,\varphi)=0,\nonumber\\
&&R(\lambda,\varphi)=P_mK_2Q_m(\mu^{-1}-Q_mK_2Q_m)^{-1}Q_mB\varphi
\end{eqnarray}
We represent $R$ as $ R(\lambda,\varphi)=\mu
P_mK_2Q_mB\varphi+R_1(\lambda,\varphi),$ where
$$
R_1(\lambda,\varphi)
=P_mK_2Q_mK_2Q_m(\mu^{-1}-Q_mK_2Q_m)^{-1}Q_mB\varphi.
$$
Equation (\ref{2.2}) becomes
\begin{equation}\label{2.3}
\Big (\frac{1}{\mu}-\frac{1}{\lambda_m}\Big
)S_2\varphi-P_mB\varphi-\mu
P_mK_2Q_mB\varphi-R_1(\lambda,\varphi)=0.
\end{equation}
Taking the inner product of the left-hand side with $S_2\psi$ in
$H_2$, where $\psi\in X_m$, we get
$$
\Big (\frac{1}{\mu}-\frac{1}{\lambda_m}\Big
)(S_2\varphi,S_2\psi)_2-(B\varphi,S_2\psi)-\mu
(Q_mB\varphi,K_2S_2\psi)-(R_1(\lambda,\varphi),S_2\psi)=0
$$
Using that $Q_mK_2S_2\psi=Q_mB\psi$, we arrive at
\begin{equation}\label{2.4}
\Big (\frac{1}{\mu}-\frac{1}{\lambda_m}\Big
)(S_2\varphi,S_2\psi)_2-(B\varphi,S_2\psi)_2-\mu
(Q_mB\varphi,B\psi)_2-(L(\mu)B\varphi,B\psi)_2=0,
\end{equation}
where $L(\mu)=Q_mK_2Q_m(\mu^{-1}-Q_mK_2Q_m)^{-1}Q_m$.  By
(\ref{LUD3a}), one can check the following estimate for the last
operator
\begin{equation}\label{2.4n}
|(L(\mu)B\psi,B\psi)|\leq C(K_2Q_mB\psi,Q_mB\psi)_2=C\langle
Q_mB\psi,Q_mB\psi\rangle.
\end{equation}

\subsection{Representation of $(B\varphi,S\psi)$}

 For $\varphi,\, \psi\in X_m$, we use the
representations
\begin{equation}\label{Feb20a}
\varphi=S_2\varphi+T_2\varphi,\;\;\;\psi=S_2\psi+T_2\psi.
\end{equation}
 Then
\begin{eqnarray}\label{Feb20aa}
&&(B\varphi,S_2\psi)=\langle
S_2\varphi,S_2\psi\rangle-\lambda_m^{-1}( S_2\varphi,S_2\psi
)\nonumber\\
&&=\langle \varphi-T_2\varphi,\psi-T_2\psi\rangle-\lambda_m^{-1}(
\varphi-T_2\varphi,\psi-T_2\psi ).
\end{eqnarray}
Using the relation
$$
(\Psi_\varphi,S_2\psi)=(\Psi_\varphi,\psi)=(\varphi,S_2\psi)-\lambda_m\langle
\varphi,S_2\psi\rangle
$$
and similar equality with exchanged $\varphi$ and $\psi$, we
obtain
\begin{eqnarray*}
&&(\varphi-T_2\varphi,\psi-T_2\psi )=(T_2\varphi,T_2\psi
)-(\varphi,\psi)+(\varphi,S_2\psi)+(S_2\varphi,\psi)\\
&&=(T_2\varphi,T_2\psi
)-\lambda_m\langle\varphi,\psi\rangle+(\Psi_\varphi,\psi)+\lambda_m\langle
\varphi,S_2\psi\rangle\\
&&+(\varphi,\Psi_\psi)+\lambda_m\langle S_2\varphi,\psi\rangle .
\end{eqnarray*}
Replacing  the last term in (\ref{Feb20aa}) according to this
formula, we derive from (\ref{Feb20aa})
\begin{equation}\label{Feb20ab}
(B\varphi,S_2\psi)= \langle
T_2\varphi,T_2\psi\rangle-\lambda_m^{-1}\Big (( T_2\varphi,T_2\psi
)+(\Psi_\varphi,\psi)+(\varphi,\Psi_\psi)\Big ).
\end{equation}

\begin{kor} Let $H_2\subset H_1$ and let $\lambda_1$ be the first
eigenvalue of {\rm (\ref{1.1})} and $\varphi$ be a corresponding
eigenfunction. Then $\Psi_\varphi=\Psi_\psi=0$ and from {\rm
(\ref{2.4})} and {\rm (\ref{Feb20ab})} it follows that
\begin{equation}\label{Nov5aa}
1-\frac{\lambda_1}{\mu_1}\leq
\frac{||T\varphi||^2}{||S\varphi||^2},
\end{equation}
where $\mu_1$ is the first eigenvalue of {\rm (\ref{1.2})}.
\end{kor}

\subsection{Representation of $(Q_mB\varphi,B\psi)$}\label{Sect4y}

For $\varphi\in X_M$ and $w\in H_2$, we have the following
representation
\begin{eqnarray*}
&&(B\varphi,w)=(K_2
S_2\varphi,w)-\frac{1}{\lambda_m}(S_2\varphi,w)
=\langle\varphi-T_2\varphi,w\rangle-\frac{1}{\lambda_m}(\varphi,w)\\
&&=-\frac{1}{\lambda_m}(\Psi_\varphi,w)-\langle
T_2\varphi,w\rangle.
\end{eqnarray*}
This implies
\begin{equation}\label{Sept24a}
B\varphi=K_2\Phi_ \varphi-\lambda_m^{-1}\Psi_\varphi,
\end{equation}
where $\Phi_\varphi$ is the orthogonal with respect to the inner
product in ${\cal H}$ projection of $T_2\varphi$ onto ${\cal
H}_2$. Using (\ref{Sept24a}), we get the desired representation
\begin{equation}\label{Sept24c}
(Q_mB\varphi,B\psi
)=\frac{1}{\lambda_m^2}(\Psi_\varphi,\Psi_\psi)+b_0(\varphi,\psi),
\end{equation}
where
\begin{eqnarray}\label{Feb20b}
&&b_0(\varphi,\psi)=-\frac{1}{\lambda_m^2}(P_m\Psi_\varphi,\Psi_\psi)
+(Q_mK_2\Phi_\varphi,K_2\Phi_\psi)\nonumber\\
&& -\frac{1}{\lambda_m}\big
(\langle\Phi_\varphi,Q_m\Psi_\psi\rangle+\langle
Q_m\Psi_\varphi,\Phi_\psi\rangle\big ).
\end{eqnarray}
Let us estimate the form $b_0$. Since
$$
||K_2\Phi_ \varphi||\leq C\mid \Phi_ \varphi\mid\leq C\mid
T_2\varphi\mid,
$$
it follows from (\ref{Sept24a}) and (\ref{4.3}) that
\begin{equation}\label{Sept24bz}
||\Psi_\varphi||\leq C\big (\sigma^{1/2}||\varphi||+\mid
T_2\varphi\mid\big )\;\;\;\mbox{for $\varphi\in X_m$.}
\end{equation}
Using that $T_2\varphi=(S_1-S_2)\varphi$, we have that $\mid
T_2\varphi\mid^2\leq\sigma ||\varphi||^2$. Therefore,
(\ref{Sept24bz}) implies
\begin{equation}\label{Sept24b}
||\Psi_\varphi||\leq C\sigma^{1/2}||\varphi||\;\;\;\mbox{for
$\varphi\in X_m$.}
\end{equation}

Let $\Upsilon_k$, $k=1,\ldots,J_m$, be an orthogonal basis in
$SX_m$ and let $\Upsilon_k=S\varphi_k$, $\varphi_k\in X_m$. Then
\begin{equation}\label{Sept20am}
(P_m\Psi_\varphi,\Psi_\psi)=\sum_{k=1}^{J_m}(\Psi_\varphi,S\varphi_k)
\,(S\varphi_k,\Psi_\psi).
\end{equation}
Using definitions of $S$, $K_2$ and $\Psi_\varphi$, one can verify
that
\begin{equation}\label{Sept24da}
S\varphi=\lambda_m K_2\varphi^*+\Psi_\varphi\;\;\;\mbox{for
$\varphi\in X_m$,}
\end{equation}
where $\varphi^*$ is orthogonal in ${\cal H}$ projection of
$\varphi$ on ${\cal H}_2$. Therefore,
$$
(\Psi_\varphi,S\varphi_k)=\lambda_m
\langle\Psi_\varphi,\varphi^*_k\rangle
+(\Psi_\varphi,\Psi_{\varphi_k}).
$$
Applying  inequalities (\ref{Sept24b}) for estimation of and $||
\Psi_{\varphi_k}||$, we get
$$
|(\Psi_\varphi,S\varphi_k)|\leq C\big
(\mid\Psi_\varphi\mid+\sigma^{1/2}||\Psi_\varphi||\big ).
$$
Therefore, it follows from (\ref{Sept20am}) that
\begin{equation}\label{Sept24d}
|(P_m\Psi_\varphi,\Psi_\psi)|\leq C\big (\sigma
(||\Psi_\varphi||^2+||\Psi_\psi||^2)+\mid\Psi_\varphi\mid^2+\mid\Psi_\psi\mid^2\big
).
\end{equation}
Since $K_2$ is a bounded operator from ${\cal H}_2$ to $H_2$, we
have
\begin{equation}\label{Sept20a}
(Q_mK_2\Phi_\varphi,K_2\Phi_\varphi)\leq(K_2\Phi_\varphi,K_2\Phi_\varphi)\leq
C\langle\Phi_\varphi,\Phi_\varphi\rangle\leq C\mid T\varphi\mid^2.
\end{equation}
 Similarly,
\begin{equation}\label{Sept20b}
\Big
|\langle\Phi_\varphi,\Psi_\psi\rangle+\langle\Psi_\varphi,\Phi_\psi\rangle\Big
|\leq \mid T\varphi\mid^2+\mid T\psi\mid^2+\mid
\Psi_\varphi\mid^2+\mid \Psi_\psi\mid^2.
\end{equation}
Combining (\ref{Sept24d})--(\ref{Sept20b}), we obtain
\begin{equation}\label{Sept24e}
b_0(\varphi,\psi)\leq C(\wp(\varphi)+\wp(\psi)),
\end{equation}
where
$$
\wp(\varphi)=\mid T_2\varphi\mid^2+\mid \Psi_\varphi\mid^2 +\sigma
||\Psi_\varphi||^2.
$$
We note that from the definition of $\rho$, see (\ref{Sept20c}),
it follows that
\begin{equation}\label{Feb10azzz}
\rho= \max_{\varphi\in X_m,||\varphi||=1}\wp(\varphi).
\end{equation}

\subsection{Estimate of $\langle
Q_mB\varphi,Q_m B\psi\rangle$}

Let us consider first the term $\langle P_mB\varphi,B\psi\rangle$.
Using the basis $\Upsilon_k=S\varphi_k$, $k=1,\ldots,J_m$,
introduced in Sect. \ref{Sect4y}, we have
\begin{equation}\label{Sept24f}
\langle
P_mB\varphi,B\psi\rangle=\sum_{k=1}^{J_m}(B\varphi,S\varphi_k)\langle
S\varphi_k,B\psi\rangle .
\end{equation}
Applying representation (\ref{Sept24da}), we get
$$
(B\varphi,S\varphi_k)=\lambda_m\langle B\varphi,
\varphi_k^*\rangle+(B\varphi,\Psi_{\varphi_k}),
$$
which together with (\ref{Sept24b}) and (\ref{Sept24a}) leads
$$
(B\varphi,S\varphi_k)\leq  C\big (\mid T_2\varphi\mid
+\mid\Psi_\varphi\mid+\sigma^{1/2}||\Psi_\varphi||\big ).
$$
Similarly,
$$
|\langle S\varphi_k,B\psi\rangle|\leq C\big (\mid T_2\psi\mid
+\mid\Psi_\psi\mid+\sigma^{1/2}||\Psi_\psi||\big ).
$$
Applying these for estimating the right-hand side of
(\ref{Sept24f}), we get
\begin{equation}\label{Sept24g}
|\langle P_mB\varphi,B\psi\rangle|\leq C(\wp(\varphi)+\wp(\psi)).
\end{equation}
Similar considerations give the estimate
\begin{equation}\label{Sept24ga}
|\langle P_mB\varphi,P_mB\psi\rangle|\leq
C(\wp(\varphi)+\wp(\psi)).
\end{equation}
Furthermore, using (\ref{Sept24a}), we get
\begin{equation}\label{Sept24gaa}
|\langle B\varphi,B\psi\rangle|\leq C (\mid T_2\varphi\mid^2+\mid
T_2\psi\mid^2+\mid\Psi_\varphi\mid^2+\mid\Psi_\psi\mid^2).
\end{equation}

Applying (\ref{Sept24g})--(\ref{Sept24gaa}) for estimating the
right-hand side in (\ref{2.4n}), we get
\begin{equation}\label{Sept24h}
 |(L(\mu)B\varphi,B\psi)|\leq  C(\wp(\varphi)+\wp(\psi)).
\end{equation}

\subsection{Proof of Theorem \protect\ref{Th1}}

Let $\mu$ satisfies (\ref{Sept28b}). We put
$\widehat{\tau}=\mu^{-1}-\lambda_m^{-1}$. Then
$\mu=\lambda_m-\lambda_m^2\widehat{\tau}+O(\sigma)$. Therefore, from
(\ref{Sept24c}) and (\ref{Sept24e}) it follows that
\begin{equation}\label{Sept29a}
\mu(Q_mB\varphi,B\psi
)=\frac{1-\lambda_m\widehat{\tau}}{\lambda_m}(\Psi_\varphi,\Psi_\psi)+b_2(\varphi,\psi),
\end{equation}
where
$$
b_2(\varphi,\psi)=\mu
b_0(\varphi,\psi)+\widehat{\tau}\frac{\lambda_m-\mu}{\lambda_m}(\Psi_\varphi,\Psi_\psi)
$$
and $b_0$ is given by (\ref{Feb20b}). Using (\ref{Sept24e}) and
(\ref{Sept28b}), we get
\begin{equation}\label{Sept29b}
b_2(\varphi,\psi)\leq C(\wp(\varphi)+\wp(\psi)).
\end{equation}

Using  (\ref{Feb20ab}), (\ref{Sept29a}) and (\ref{Sept29b}) together
with (\ref{Sept24h}), we derive from (\ref{2.4})
\begin{eqnarray}\label{Sept27d}
&&\widehat{\tau}\big ((S\varphi , S\psi
)+(\Psi_\varphi,\Psi_\psi)\big )=\frac{1}{\lambda_m}\Big
((\Psi_\varphi,\Psi_\psi)-(T\varphi,T\psi)\nonumber\\
&&-(\Psi_\varphi,\psi)-(\varphi,\Psi_\psi)\Big ) +b(\varphi,\psi)
\;\;\;\mbox{for all $\psi\in X_m$,}
\end{eqnarray}
where
$$
b(\varphi,\psi)=b_2(\varphi,\psi)+(L(\mu)B\varphi,B\psi)+\langle
T_2\varphi,T_2\psi\rangle.
$$
Due to  (\ref{Sept29b}) and (\ref{Sept24h}), the form $b$ is subject
to
\begin{equation}\label{Sept27a}
|b(\varphi,\psi)|\leq C(\wp(\varphi)+\wp(\psi)).
\end{equation}
Moreover, by (\ref{Sept24b})
$$
|(\Psi_\varphi,\Psi_\psi)|\leq c\sigma (||\varphi||^2+
||\psi||^2).
$$
From (\ref{4.1}) it follows that $||\varphi||^2\leq
(1+c\sigma^{1/2})||S\varphi||^2$. Therefore
\begin{equation}\label{Okt12a}
|(\Psi_\varphi,\Psi_\psi)|\leq c\sigma (||S\varphi||^2+
||S\psi||^2).
\end{equation}

Let $\mu_1^{-1},\ldots,\mu_{J_m}^{-1}$ be the eigenvalues of the
problem (\ref{1.2}) and let $U_1,\ldots,U_{J_m}$ be corresponding
eigenvectors. We assume that the eigenvectors are chosen to satisfy
$(U_j,U_k)=\delta_{j,k}$, where  $\delta_{j,k}$ is the Kronecker
delta. Since $P_m$ is the orthogonal projector with the image
$SX_m$, we can represent $P_mU_j$ as $P_mU_j=SV_j$, where $V_j\in
X_m$. According to (\ref{Okt13a}),
\begin{equation}\label{Okt12c}
(SV_j,SV_k)=\delta_{jk}+{\cal O}(\sigma)\;\;\;\mbox{for
$j,k=1,\ldots,J_m$.}
\end{equation}
Moreover according to the reduction from Sect. \ref{S2},
$\hat{\tau}_j=\mu_j^{-1}-\lambda_m^{-1}$, $j=1,\cdots,J_m$, is an
eigenvalue to problem (\ref{Sept27d}) and $\varphi=V_j$ is the
corresponding eigenvector.

Denote by $\tau_k$, $k=1,\ldots,J_m$, the eigenvalue of the
finite-dimensional problem (\ref{Eig19av1aa}) and by
$\Phi_1,\ldots,\Phi_{J_m}$ corresponding eigenvectors from $X_m$,
which satisfies the bi-orthogonality condition
\begin{equation}\label{Okt12ab}
(S\Phi_j,S\Phi_k)=\delta_{jk}\;\;\;\mbox{for $j,k=1,\ldots,J_m$.}
\end{equation}
Let us show that for each $j=1,\dots,J_m$ there exists $k=k(j)$,
$1\leq k\leq\dots,J_m$ such that
\begin{equation}\label{Okt12d}
(SV_j,S\Phi_{k(j)})\geq c_*,
\end{equation}
where $c_*$ is a positive constant depending on $J_m$. Moreover
the function $k(j)$ is isomorphism. In order to prove these facts
we consider the matrix $A=\{A_{jk}\}$, where
$A_{jk}=(SV_j,S\Phi_k)$, $j,k=1,\ldots,J_m$. Let us show that
\begin{equation}\label{Okt12e}
\det A=1+{\cal O}(\sigma^{1/2})
\end{equation}
Indeed, let $\nu$ be an eigenvalue of the matrix $A$ and ${\bf
a}=(a_1,\ldots,a_{J_m})$ be corresponding eigenvector with the
norm $1$. Then
$$
\sum_k(SV_j,S\Phi_k)a_k=\nu a_j\;\;\mbox{for $j=1,\ldots,J_m$,}
$$
or equivalently $(SV_j,S\Phi)=\nu a_j$, where
$\Phi=\overline{a}_1\Phi_1+\cdots+\overline{a}_{J_m}\Phi_{J_m}$.
Clearly, $||S\Phi||=1$. We chose the constants $b_j$,
$j=1,\ldots,J_m$, such that $\sum b_jV_j=\Phi$. Clearly, the norm of
the vector ${\bf b}=(b_1,\ldots,b_{J_m})$ is equal to $1+{\cal
O}(\sigma)$. Then
$$
\sum_{j=1}^{J_m}b_j(SV_j,S\Phi)=(S\Phi,S\Phi)=\nu
\sum_{j=1}^{J_m}a_jb_j.
$$
This implies $1\leq |\nu|(1+{\cal O}(\sigma))$. Since the last
relation is valid for all eigenvalues of $A$ we obtain
(\ref{Okt12e}). Therefore, there exists an isomorphism $k(j)$ such
that the equality  (\ref{Okt12d}) is valid for $j=1,\cdots,J_m$.
After the re-numeration of eigenvalues $\tau_j$ and corresponding
eigenvectors we can assume that the relations
\begin{equation}\label{Okt12dj}
(SV_j,S\Phi_j)\geq c_*,\;\;\; j=1,\ldots,J_m,
\end{equation}
hold.

Choosing $\varphi=V_j$ and $\psi=\Phi_j$ in (\ref{Sept27d}) we
obtain
\begin{equation}\label{Nov1a}
(\widehat{\tau}_k-\tau_k)(SV_j,S\Phi_j)=b(V_j,\Phi_j)-\widehat{\tau}_j(\Psi_{V_j},\Psi_{\Phi_j}).
\end{equation}
Using relations (\ref{Okt12dj}), (\ref{Sept27a})
 and (\ref{Sept24b})
together with definition (\ref{Sept20c}) of $\rho$, we derive from
(\ref{Nov1a})
\begin{equation}\label{Nov1b}
|\widehat{\tau}_k-\tau_k|\leq C(\max_{\psi\in
X_m,||\psi||=1}\wp(\psi)+|\widehat{\tau}_k|\sigma ).
\end{equation}
By (\ref{Feb10azzz}), we get
\begin{equation}\label{Nov1bz}
|\widehat{\tau}_k-\tau_k|\leq C(\rho+|\widehat{\tau}_k|\sigma ).
\end{equation}
 This implies
\begin{equation}\label{Nov1ba}
|\widehat{\tau}_k-\tau_k|\leq C(\rho+|\tau_k|\sigma )
\end{equation}
and hence (\ref{Sept20d}).

\section{Application to a second order elliptic
equation}\label{SectApl}

Here we consider the spectral problems (\ref{TTN1}) and
(\ref{TTN2}) generated by the bi-linear form (\ref{OperA}).
Instead of (\ref{1.1azz}) it is sufficient to check inequality
(\ref{1.1a}).

\subsection{Local perturbation of the boundary}\label{SectLoca}

{\em Constant $\sigma_*$}.
 Let  $\varepsilon$ be a small positive number. We
assume that there exists a point $x_0$ in $\partial\Omega_1$
such that
\begin{equation}\label{Jan1a}
\Omega_1\setminus
B_\varepsilon(x_0)\subset\Omega_2\subset\Omega_1\cup
B_\varepsilon(x_0).
\end{equation}
We  denote the domain $\Omega_1\cup B_\varepsilon(x_0)$ by
$\Omega_\varepsilon^+(x_0)$. We also assume that for
$u\in\W^{1,2}(\Omega_\varepsilon^+(x_0))$
\begin{equation}\label{Jan8a}
\int_{B_{q\varepsilon}(x_0)\cap
\Omega_\varepsilon^+(x_0)}|u|^2dx\leq
C\varepsilon^2\int_{B_{q\varepsilon}(x_0)\cap
\Omega_\varepsilon^+(x_0)}|\nabla u|^2dx
\end{equation}
for a certain $q>1$ independent of $\varepsilon$. Here the
constant $C$ may depend on $\Omega_1$, $n$, $q$ and the
ellipticity constant $\nu$.

 Let  ${\cal
Z}_\varepsilon(x_0)$ be subspace of function in
$\W^{1,2}(\Omega_\varepsilon^+)$ subject to
\begin{equation}\label{Jan1b}
\sum_{i,j=1}^n\int A_{ij}(x)\partial_{x_j} u\partial_{x_i}
wdx=0\;\;\;\mbox{for all $w\in \W^{1,2}(\Omega_1\setminus
B_\varepsilon(x_0))$.}
\end{equation}
In what follows we will omit the summation sign in  formulas like
(\ref{Jan1b}).

\begin{lem}\label{Lem10a} {\rm (i)} Let $u\in {\cal Z}_{q^{1/2}\varepsilon}(x_0)$.
Then
\begin{equation}\label{Jan8b}
\int_{\Omega_\varepsilon^+(x_0)} |\nabla u|^2dx\leq
C_1\int_{B_{q\varepsilon}(x_0)\cap (\Omega_\varepsilon^+(x_0)}
|\nabla u|^2dx.
\end{equation}

\noindent {\rm (ii)} Let $q_*\in [1,q^{1/2}]$ and let
$T_{q_*\varepsilon}(x_0)$ be orthogonal projector from
$\W^{1,2}(\Omega_\varepsilon^+(x_0))$ onto ${\cal
Z}_{q_*\varepsilon}(x_0)$. Then
\begin{equation}\label{Jan8bk}
\int_{\Omega_\varepsilon^+(x_0)} |\nabla
(T_{q_*\varepsilon}(x_0)u)|^2dx\leq
C_1\int_{B_{q\varepsilon}(x_0)\cap (\Omega_\varepsilon^+(x_0)}
|\nabla u|^2dx.
\end{equation}
The constant $C_1$ in {\rm (i)} and {\rm (ii)} may depend on $n$,
$\nu$, $q$ and $\Omega_1$.
\end{lem}

\begin{proof} We introduce a smooth function $\eta=\eta(t)$  which is equal to
$1$ for $t<q^{1/2}$ and to $0$ for $t>q$ and let $\eta_\varepsilon
(x)=\eta (|x-x_0|/\varepsilon)$.

(i) Put $u_\varepsilon=u-\eta_\varepsilon u$. One can check that
$u_\varepsilon\in\W^{1,2}(\Omega_1\setminus
B_{q^{1/2}\varepsilon}(x_0))$. From (\ref{Jan1b}) it follows that
\begin{equation}\label{Jan1c}
\int A_{ij}\partial_{x_j} u_\varepsilon\partial_{x_j}
wdx\!\!=\!\!-\!\!\int A_{ij}\partial_{x_j} (\eta_\varepsilon
u)\partial_{x_i} wdx\;\;\;\mbox{for all $w\in
\W^{1,2}(\Omega_1\setminus B_{q^{1/2}\varepsilon}(x_0))$.}
\end{equation}
Taking $w=u_\varepsilon$ in (\ref{Jan1c}) we obtain
\begin{equation}\label{Jan1d}
\int A_{ij}\partial_{x_j} u_\varepsilon\partial_{x_i}
u_\varepsilon dx=-\int A_{ij}\partial_{x_j} (\eta_\varepsilon
u)\partial_{x_i} u_\varepsilon dx,
\end{equation}
which together with (\ref{Jan3a}) implies
$$
\int |\nabla u_\varepsilon|^2dx\leq \nu^{-4}\int |\nabla
(\eta_\varepsilon u)|^2dx.
$$
Using (\ref{Jan8a}), we derive  the following estimate from the
last inequality:
$$
\int_{\Omega_1} |\nabla u_\varepsilon|^2dx\leq
C_1\int_{B_{q\varepsilon}(x_0)\cap (\Omega_\varepsilon^+(x_0)}
|\nabla u|^2dx.
$$
Similar estimate for $\eta_\varepsilon u$ follows from
(\ref{Jan8a}), which together with the estimate for
$u_\varepsilon$ leads to (\ref{Jan8b}).

(ii) Let $U_\varepsilon=T_{q_*\varepsilon}(x_0)u-\eta_\varepsilon
u$. Then  $U_\varepsilon(x)\in\W^{1,2}(\Omega_1\setminus
B_{q_*\varepsilon}(x_0))$ and $U_\varepsilon+\eta_\varepsilon u\in
{\cal Z}_{q_*\varepsilon}(x_0)$. Therefore,
\begin{equation*}
\int A_{ij}\partial_{x_j}U_\varepsilon\partial_{x_j}wdx=-\int
A_{ij}\partial_{x_j}(\eta_\varepsilon
u)\partial_{x_j}wdx\;\;\mbox{for all
$w\in\W^{1,2}(\Omega_1\setminus B_{q_*\varepsilon}(x_0))$.}
\end{equation*}
Taking here $w=U_\varepsilon$ and using  H\"older inequality along
with inequalities (\ref{Jan3a}), we obtain
$$
\int |\nabla U_\varepsilon|^2dx\leq \nu^{-4}\int |\nabla
(\eta_\varepsilon u)|^2dx\leq C\int_{B_{q\varepsilon}(x_0)\cap
(\Omega_\varepsilon^+(x_0)} |\nabla u|^2dx.
$$
To get the last inequality we applied (\ref{Jan8a}). The estimate
of $\eta_\varepsilon u$ by the right-hand side of the last
inequality follows from (\ref{Jan8a}). Since
$T_{q_*\varepsilon}(x_0)u=\eta_\varepsilon u+U_\varepsilon$, the
above two estimates give (\ref{Jan8bk}). The proof is complete.
\end{proof}

Now we are in position to prove the following

\begin{prop}\label{PropJan13a} There exists a function $\sigma_*=\sigma_*(\varepsilon)$
such that $\sigma_*(\varepsilon)\to 0$ as $\varepsilon\to 0$ and
for all $u\in {\cal Z}_\varepsilon(x_0)$
\begin{equation}\label{Jan4gg}
\int_{\Omega_\varepsilon^+(x_0)}|u|^2dx\leq \sigma
\int_{\Omega_\varepsilon^+(x_0)}|\nabla u|^2dx.
\end{equation}
\end{prop}
\begin{proof}
 Let $\lambda_k$, $k=1,\ldots,$ be eigenvalues of the problem
(\ref{TTN1}) and let $\varphi_k$ be corresponding eigenfunctions
normalized by $||\varphi||_{W^{1,2}(\Omega_1)}=1$. We represent
$u$ as $u=\eta_{q^{-1/2}\varepsilon} u+u_\varepsilon$. Then
$u_\varepsilon\in\W^{1,2}(\Omega_1)$ and $T_{q^{1/2}\varepsilon}
u_\varepsilon=u_\varepsilon$. Therefore we may represent $u$ as
$$
u_\varepsilon=\sum_{k=N+1}^\infty(u_\varepsilon,\varphi_k)\varphi_k
+\sum_{k=1}^N(u_\varepsilon,T_{q^{1/2}\varepsilon}\varphi_k)\varphi_k.
$$
 Since $\{\varphi_k\}_{k\geq 1}$ is an orthogonal basis in
$L^2(\Omega_1)$ also and
$||\varphi_k||_{L^2(\Omega_1)}^2=\lambda_k^{-1}$, we have
\begin{eqnarray*}
&&||u_\varepsilon||_{L^2(\Omega_1)}^2\leq
\frac{1}{\lambda_{N+1}}||u_\varepsilon||_{W^{1,2}(\Omega_1)}^2+\sum_{k=1}^N\frac{1}{\lambda_k}
|(u_\varepsilon,T_{q^{1/2}\varepsilon}\varphi_k)|^2\\
&&\leq\Big
(\frac{1}{\lambda_{N+1}}+\sum_{k=1}^N\frac{1}{\lambda_k}||T_{q^{1/2}\varepsilon}
\varphi_k||_{W^{1,2}(\Omega_1)}^2\Big
)||u_\varepsilon||_{W^{1,2}(\Omega_1)}^2.
\end{eqnarray*}
Using (\ref{Jan8b}) we get
$$
||u_{\varepsilon}||_{L^2(\Omega_1)}^2\leq
(\frac{1}{\lambda_{N+1}}+C\sum_{k=1}^N\frac{1}{\lambda_k}
\int_{\Omega_1\setminus\Omega_{q\varepsilon}}|\nabla\varphi_k|^2dx\Big
)||u_\varepsilon||_{W^{1,2}(\Omega_1)}^2
$$
Let
$$
\sigma_1=\inf \Big
(\frac{1}{\lambda_{N+1}}+C\sum_{k=1}^N\frac{1}{\lambda_k}
\int_{\Omega_1\setminus\Omega_{q\varepsilon}}|\nabla\varphi_k|^2dx\Big
).
$$
then $\sigma_1(\varepsilon)\to 0$ as $\varepsilon\to 0$ and by the
last inequality for $u_\varepsilon$ combined with (\ref{Jan8a}),
we obtain
$$
||u_\varepsilon||_{L^2(\Omega_1)}^2\leq\sigma_1||\nabla
(u-\eta_{q^{-1/2}\varepsilon}u)||_{W^{1,2}(\Omega_\varepsilon^+)}\leq
C_1(\sigma_1+\varepsilon^2)||\nabla
u||^2_{W^{1,2}(\Omega_\varepsilon^+)}.
$$
Using (\ref{Jan8a}) again, we get
$$
||\eta_{q^{-1/2}\varepsilon}
u||_{L^2(\Omega_\varepsilon^+(x_0))}^2\leq C_2\varepsilon^2
\int_{\Omega_\varepsilon^+(x_0)\setminus
B_{q\varepsilon}(x_0)}|\nabla u|^2dx.
$$
Now setting $\sigma=2C_1(\sigma_1+\varepsilon)+2C_2\varepsilon^2$
and using the representation $u=\eta_{q^{-1/2}\varepsilon}
u+u_\varepsilon$ we arrive at (\ref{Jan4gg}).
\end{proof}

\begin{Rem} In the case $\Omega_2\subset\Omega_1$ {\rm Lemma \ref{Lem10a}}
and {\rm Proposition \ref{PropJan13a}} are valid if
$\Omega_\varepsilon(x_0)$ is replaced by $\Omega_1$ in their
formulations and in {\rm (\ref{Jan8a})}.

\end{Rem}

\begin{Rem} Let  several points $x_1,\ldots, x_m$ be given on the
boundary $\partial\Omega_1$. Let also $\varepsilon$ be a small
positive number, $q>1$ and inequality {\rm (\ref{Jan8a})} be valid
for all points $x_1,\ldots,x_m$. Then {\rm Lemma \ref{Lem10a}} and
{\rm Proposition \ref{PropJan13a}} remain true for perturbation
$\Omega_2$ subject to
$$
\Omega_1\setminus
\bigcup_{j=1}^mB_\varepsilon(x_j)\subset\Omega_2\subset\Omega_1\cup
\bigcup_{j=1}^mB_\varepsilon(x_j)
$$
if we replace $\Omega_\varepsilon^+(x_0)$ by $\Omega_1\cup
\bigcup_{j=1}^mB_\varepsilon(x_j)$ and $B_{q\varepsilon}(x_0)$ by
$\bigcup_{j=1}^mB_{q\varepsilon}(x_j)$ in their formulations.
\end{Rem}

{\em Estimates of the function $u=T\varphi$, $\varphi\in X_m$.}
Since $TT_\varepsilon=T_\varepsilon T=T$, we have that
\begin{equation}\label{Jan30a}
\int A_{ij}\partial_{x_j}(T\varphi)\partial_{x_i}(T\varphi)dx\leq
\int
A_{ij}\partial_{x_j}(T_\varepsilon\varphi)\partial_{x_i}(T_\varepsilon\varphi)dx.
\end{equation}
This together with Lemma \ref{Lem10a}(ii) gives
\begin{equation}\label{Nov9b}
||\nabla (T\varphi)||_{L^2(\Omega_2)}\leq
c||\nabla(T\varphi)||_{L^2(B_{q\varepsilon}
(x_0)\cap\Omega_\varepsilon^+)}.
\end{equation}

{\em Estimate of the function $\Psi=\Psi_\varphi$.} We seek $\Psi$
in the form $\Psi=\eta_\varepsilon(x)+v$, where $\eta=\eta(t)$ is
a smooth function equals $1$ for $t<q_1=(1+q)/2$ and $0$ for
$t>q_2=(1+q_1)/2$. Then the function $v$ belongs to
$\W^{1,2}(\Omega_2)$ and satisfies
\begin{equation}\label{Nov10a}
(v,w)=-\lambda_m\langle \varphi,\eta_\epsilon
w\rangle+\int_{\Omega_2}A_{ij}\partial_{x_j}
v\,w\partial_{x_i}\eta_\epsilon
dx-\int_{\Omega_2}vA_{ij}\partial_{x_j}\eta_\varepsilon\,\partial_{x_i}
wdx.
\end{equation}
Applying H\"older's inequality to the left-hand side of
(\ref{Nov10a}) and using then (\ref{Jan8a})  we arrive at
\begin{equation}\label{Nov10c}
||\nabla v||_{L^2(\Omega_2)}\leq
c||\nabla\varphi||_{L^2(\Omega_1\cap B_{q\delta}(x_2))}.
\end{equation}
Applying (\ref{Jan8a}), we get an estimate of $\eta_\varepsilon$
by the right-hand side of (\ref{Nov10c}). Combining these two
estimates, we obtain
\begin{equation}\label{Nov10d}
||\nabla \Psi||_{L^2(\Omega_2)}\leq
c||\nabla\varphi||_{L^2(\Omega_1\cap B_{q\varepsilon}(x_0))}.
\end{equation}

Using estimates (\ref{Nov9b}), (\ref{Nov10d}) and Corollary
\ref{Kor2}, we arrive at
\begin{kor} Under the assumptions on $\Omega_1$ and $\Omega_2$ of this
section, the following estimate for the eigenvalues of the
problems {\rm (\ref{TTN1})} and {\rm (\ref{TTN2})} holds
$$
|\lambda_m^{-1}-\mu_k^{-1}|\leq C\max_{\varphi\in X_m,
||\varphi||=1}||\nabla\varphi||_{L^2(\Omega_1\cap
B_{q\varepsilon}(x_0))}^2,
$$
where $\mu_1,\ldots,\mu_{J_m}$ are eigenvalues of {\rm
(\ref{TTN2})} located near $\lambda_m^{-1}$, see {\rm Proposition
\ref{Pr1}}.
\end{kor}

\subsection{Global perturbation of the boundary}\label{SectGla}

 Here we consider perturbations of $\Omega_1$ located near
the boundary. Let $\varepsilon$ be a small positive number. We
introduce the sets
$$
\Omega_\varepsilon=\{ x\in\Omega_1\,:\, \dist (x,\partial\Omega_1)
>\varepsilon\}\;\;\mbox{and}\;\;
\Omega_\varepsilon^+=\{y\in\Bbb R^n\,:\, \dist
(y,\Omega_1)<\varepsilon).
$$

\noindent We assume that
\begin{equation}\label{Jan17a}
\Omega_\varepsilon\subset\Omega_2\subset\Omega_\varepsilon^+.
\end{equation}
and that for all $x_0\in\partial\Omega_1$ and
$u\in\W^{1,2}(\Omega_\varepsilon^+)$ the inequality
\begin{equation}\label{Jan11a}
\int_{B_{q\varepsilon}(x_0)\cap \Omega_\varepsilon^+}|u|^2dx\leq
C\varepsilon^2\int_{B_{q\varepsilon}(x_0)\cap
\Omega_\varepsilon^+}|\nabla u|^2dx
\end{equation}
holds with a certain $q>1$ independent of $\varepsilon$.

 Let ${\cal Z}_\delta$ be subspace of
function in $\W^{1,2}(\Omega_\varepsilon^+)$ subject to
\begin{equation}\label{Jan1bz}
\sum_{i,j=1}^n\int A_{ij}(x)\partial_{x_j} u\partial_{x_i}
wdx=0\;\;\;\mbox{for all $w\in
\W^{1,2}(\Omega_\varepsilon^+\setminus \Omega_\delta)$.}
\end{equation}

\begin{lem}\label{Lem12a} {\rm (i)} Let $q_1\in (1,q)$ and
$u\in\W^{1,2}(\Omega_\varepsilon^+)$. Then
\begin{equation}\label{Jan8bb}
\int_{\Omega_\varepsilon^+\setminus\Omega_{q_1\varepsilon}}
|\nabla u|^2dx\leq C_1\varepsilon^2\int_{
\Omega_\varepsilon^+\setminus \Omega_{q\varepsilon}} |\nabla
u|^2dx.
\end{equation}

{\rm (ii)} Let $u\in {\cal Z}_{q_0\varepsilon}$ with $q_0\in
(1,q)$. Then
\begin{equation}\label{Jan8bz}
\int_{\Omega_\varepsilon^+} |\nabla u|^2dx\leq C_1\int_{
\Omega_\varepsilon^+\setminus \Omega_{q\varepsilon}} |\nabla
u|^2dx.
\end{equation}

{\rm (iii)} Let $q_*\in [1,q)$ and let $T_{q_*\varepsilon}$ be
orthogonal projector from $\W^{1,2}(\Omega_\varepsilon^+)$ onto
${\cal Z}_{q_*\varepsilon}$. Then
\begin{equation}\label{Jan8bkz}
\int_{\Omega_\varepsilon^+} |\nabla (T_{q_*\varepsilon}u)|^2dx\leq
C_1\int_{ \Omega_\varepsilon^+\setminus \Omega_{q\varepsilon}}
|\nabla u|^2dx,
\end{equation}
The constant $C_1$ in {\rm (i)} --{\rm (ii)} may depend on $n$,
$\nu$, $q$, $q_0$, $q_1$, $q_*$ and $\Omega_1$.
\end{lem}
\begin{proof} (i) Let us construct a set of points on
$\partial\Omega_1$ satisfying the properties a) and b) below. We
put $\alpha=\sqrt{q^2-q_1^2}$ and choose points $x_1,\ldots,x_N$
in the following way. We take an arbitrary point $x_1\in
\partial\Omega_1$. Let the points $x_1,\ldots,x_m$ have been
chosen. If there is a point on the boundary, say $x_*$, such that
$|x-x_*|>\alpha\varepsilon$ then we put $x_{m+1}=x_*$. If there
are no such point the the required set is constructed and we take
$N=m$. The above procedure leads to a finite set of points with
the following properties:

\smallskip
\noindent a). $\Omega_\varepsilon^+\setminus\Omega_{q_1\varepsilon
}\subset \bigcup_{k=1}^NB_{q\varepsilon}(x_k)$,

\smallskip
\noindent b). There is a integer $M$, depending only on $n$ and
$q_1$, $q$, such that every $x\in\Bbb R^n$ may belong at most to
$M$ balls $B_{q\varepsilon}(x_k)$, $k=1,\ldots,N$.

\smallskip
Using a), (\ref{Jan8a}) and then b), we get
\begin{eqnarray*}
&&\int_{\Omega_\varepsilon^+\setminus\Omega_{q_1\varepsilon
}}|u|^2dx\leq \int_{\bigcup_{k=1}^NB_{q\varepsilon}(x_k)}|u|^2dx\\
&&\leq
C\varepsilon^2\int_{\bigcup_{k=1}^NB_{q\varepsilon}(x_k)}|\nabla
u|^2dx\leq CM
\int_{\Omega_\varepsilon^+\setminus\Omega_{q\varepsilon }}|\nabla
u|^2dx,
\end{eqnarray*}
which leads to (\ref{Jan8bb}).

(ii). Let $\eta=\eta(t)$ be a smooth function which is equal to
$1$ for $t<q_0$ and $0$ for $t>q_1=(q_0+q)/2$. It is clear that
$1<q_0<q_1<q$. Let also $\zeta_\varepsilon (x)=\eta
(d(x)/\varepsilon)$, where $d(x)=\max_k |x-x_k|$. We represent $u$
as $u=\zeta_\varepsilon u+u_\varepsilon$. Then $u_\varepsilon$
belongs to $\W^{1,2}(\Omega_\varepsilon^+\setminus
\Omega_{q_1\varepsilon})$ and satisfies
$$
\sum_{i,j=1}^n\int A_{ij}(x)\partial_{x_j}
(u_\varepsilon+\zeta_\varepsilon u)\partial_{x_i}
wdx=0\;\;\;\mbox{for all $w\in
\W^{1,2}(\Omega_\varepsilon^+\setminus \Omega_{q_0\varepsilon})$.}
$$
Taking here $w=u_\varepsilon$, moving term with $u$ to the
right-hand side, using then H\"older inequality and (\ref{Jan3a}),
we obtain
$$
\int |u_\varepsilon|^2dx\leq \nu^{-4}\int |\nabla
(\zeta_\varepsilon u)|^2dx
$$
which implies, due to (\ref{Jan8bb}),
\begin{equation}\label{Jan8bs}
\int_{\Omega_\varepsilon^+} |\nabla u_\varepsilon|^2dx\leq
C_1\int_{ \Omega_\varepsilon^+\setminus \Omega_{q\varepsilon}}
|\nabla u|^2dx.
\end{equation}
 The estimate of $\zeta_\varepsilon u$ by the left-hand side of
 (\ref{Jan8bs}) follows from (\ref{Jan8bb}), which together with
 (\ref{Jan8bs}) gives (\ref{Jan8bz}).

 (iii) Let $\eta=\eta(t)$ be a smooth function which is equal to
$1$ for $t<q_0=(q_*+q)/2$ and $0$ for $t>q_1=(q_0+q)/2$. One can
check that $1<q_0<q_1<q$. Let also  $\zeta_\varepsilon (x)=\eta
(d(x)/\varepsilon)$. We represent $T_{q_*\varepsilon} u$ as
$T_{q_*\varepsilon} u=\eta_\varepsilon u+u_\varepsilon$.  Since
$T_{q_*\varepsilon} u(x)=u(x)$ for
$x\in\Omega_\varepsilon^+\setminus\Omega_{q_*\varepsilon}$ and
$$
(T_{q_*\varepsilon} u,w)=0\;\;\;\mbox{for all
$w\in\W^{1,2}(\Omega_{q_*\varepsilon})$ },
$$
we have that $u_\varepsilon\in \W^{1,2}(\Omega_{q_*\varepsilon})$
and
$$
(u_\varepsilon,w)=-(\zeta_{\varepsilon}u,w)\;\;\;\mbox{for all
$w\in \W^{1,2}(\Omega_\varepsilon)$}.
$$
We choose here $w=u_\varepsilon$ and obtain
$$
|| u_\varepsilon||^2=-(\zeta_{\varepsilon}u, u_\varepsilon).
$$
This implies
$$
||u_\varepsilon||^2_1\leq ||\zeta_{\varepsilon}u||^2_1,
$$
which leads to
\begin{equation*}
\int_{\Omega_\varepsilon^+} |\nabla u_\varepsilon|^2dx\leq
C_1\int_{ \Omega_\varepsilon^+\setminus \Omega_{q\varepsilon}}
|\nabla u|^2dx,
\end{equation*}
where we used (\ref{Jan11a}). Similar estimate of
$\eta_\varepsilon u$ by the right-hand side of the last inequality
follows from (\ref{Jan11a}). These two estimates give
(\ref{Jan8bkz}).
\end{proof}

\begin{prop} There exists a function $\sigma=\sigma(varesilon)$
such that $\sigma(\varepsilon)\to 0$ as $\varepsilon\to 0$ and for
all $u\in \W^{1,2}(\Omega_\varepsilon^+)$ satisfying {\rm
(\ref{Jan1bz})}
\begin{equation}\label{Jan4gy}
\int_{\Omega_\varepsilon^+}|u|^2dx\leq \sigma
\int_{\Omega_\varepsilon^+}|\nabla u|^2.
\end{equation}
\end{prop}
\begin{proof}
 Let $\lambda_k$, $k=1,\ldots,$ be eigenvalues of the problem
(\ref{TTN1}) and let $\varphi_k$ be corresponding eigenfunctions
normalized by $||\varphi||_1=1$. We introduce a smooth
$\eta=\eta(t)$  which is equal to $1$ for $t<q_0=(1+q)/2$ and $0$
for $t>q_1=(q_0+q)/2$. One can check that that $1<q_0<q_1<q$. Let
also $\zeta_\varepsilon (x)=\eta (d(x)/\varepsilon)$. We represent
$u$ as $u=\eta_\varepsilon u+u_\varepsilon$. Then
$u_\varepsilon(x)\in\W^{1,2}(\Omega_{q_0\varepsilon})$ and
$T_{q_1\varepsilon} u_\varepsilon=u_\varepsilon$. Therefore, we
may represent $u_\varepsilon$ as
$$
u_\varepsilon=\sum_{k=N+1}^\infty(u_\varepsilon,\varphi_k)\varphi_k
+\sum_{k=1}^N(u_\varepsilon,T_{q_1\varepsilon}\varphi_k)\varphi_k.
$$
 Since $\{\varphi_k\}_{k\geq 1}$ is an orthogonal basis in
$L^2(\Omega_1)$ also and
$||\varphi_k||_{L^2(\Omega_1)}^2=\lambda_k^{-1}$, we have
\begin{eqnarray*}
&&||u_\varepsilon||_{L^2(\Omega_1)}^2\leq
\frac{1}{\lambda_{N+1}}||u_\varepsilon||_{W^{1,2}(\Omega_1)}^2
+\sum_{k=1}^N\frac{1}{\lambda_k}|(u_\varepsilon,T_{q_1\varepsilon}\varphi_k)|^2\\
&&\leq\Big
(\frac{1}{\lambda_{N+1}}+\sum_{k=1}^N\frac{1}{\lambda_k}||T_{q_1\varepsilon}
\varphi_k||_{W^{1,2}(\Omega_1)}^2\Big
)||u_\varepsilon||_{W^{1,2}(\Omega_1)}^2.
\end{eqnarray*}
Using Lemma \ref{Lem12a}(ii) we get
$$
||u_\varepsilon||_{L^2(\Omega_1)}^2\leq
(\frac{1}{\lambda_{N+1}}+C\sum_{k=1}^N\frac{1}{\lambda_k}
\int_{\Omega_1\setminus\Omega_{q\varepsilon}}|\nabla\varphi_k|^2dx\Big
)||u_\varepsilon||_{W^{1,2}(\Omega_1)}^2
$$
Let
$$
\sigma_1=2\inf \Big
(\frac{1}{\lambda_{N+1}}+C\sum_{k=1}^N\frac{1}{\lambda_k}
\int_{\Omega_1\setminus\Omega_{q\varepsilon}}|\nabla\varphi_k|^2dx\Big
).
$$
then $\sigma_1(\varepsilon)\to 0$ as $\varepsilon\to 0$ and by the
last inequality for $u_\varepsilon$, we obtain
$$
||u_\varepsilon||_{L^2(\Omega_1)}^2\leq\sigma_1||\nabla
u_\varepsilon||_{W^{1,2}(\Omega_\varepsilon^+)}.
$$
Since $u_\varepsilon=(1-\eta_\varepsilon)u$, we get
$$
||u_\varepsilon||_{L^2(\Omega_\varepsilon^+)}^2\leq C||\nabla
u||_{W^{1,2}(\Omega_\varepsilon^+)}.
$$
Therefore
$$
||u_\varepsilon||_{L^2(\Omega_1)}^2\leq C\sigma_1||\nabla
u_\varepsilon||_{W^{1,2}(\Omega_\varepsilon^+)}.
$$
From Lemma \ref{Lem12a}(i) it follows that
$$
||\eta_\varepsilon u||_{L^2(\Omega_\varepsilon^+)}^2\leq
C_1\varepsilon^2
\int_{\Omega_\varepsilon\setminus\Omega_{q\varepsilon}^+}|\nabla
u|^2dx.
$$
Now setting $\sigma=C\sigma_1+C_1\varepsilon^2$ and using the
representation $u=\eta_\varepsilon u+u_\varepsilon$ we arrive at
(\ref{Jan4gy}).
\end{proof}

{\em Estimates for $T\varphi$.} Since
$TT_\varepsilon=T_\varepsilon T=T$, we conclude that $T$ and
$T_\varepsilon$ satisfy (\ref{Jan30a}). Now, Lemma \ref{Lem12a}
and (\ref{Jan30a}) lead to the estimate
\begin{equation}\label{Nov4b}
||T\varphi||_{H^1(\Omega_2)}\leq
C||\nabla\varphi||_{L^2(\Omega_1\setminus\Omega_{q\varepsilon}}.
\end{equation}

{\em Estimate for $\Psi_\varphi$, $\varphi\in X_m$.} Let $d=d(x)$
and $\eta_\varepsilon(x)=\eta(d(x)/\varepsilon)$ be the same as
before.  We are looking for the solution $\Psi=\Psi_\varphi$ to
equation (\ref{TTN23feb}) in the form
\begin{equation}\label{Nov8a}
\Psi=\eta_\varepsilon(x)\varphi(x)+v(x).
\end{equation}
Then $v\in\W^{1,2}(\Omega_2)$ and satisfies
\begin{equation*}
(v,w)=-\lambda\langle \eta_\varepsilon\varphi,w\rangle
+\int_{\Omega_2}(A_{ij}\partial_{x_j}\varphi\,w\partial_{x_i}\eta_\varepsilon)dx
-\int_{\Omega_2}\varphi
A_{ij}\partial_{x_j}\eta_\varepsilon\,\partial_{x_i}w)dx.
\end{equation*}
Choosing $w=v$ and using H\"older's inequality together with Lemma
\ref{Lem12a}(i), we obtain
$$
||\nabla v||_{L^2(\Omega_2)}\leq C
||\nabla\varphi||_{L^2(\Omega_1\setminus\Omega_{q\varepsilon}}
$$
Taking into account (\ref{Nov8a}), we get
\begin{equation}\label{Nov8aa}
||\Psi||_{L^2(\Omega_2)}\leq C
||\nabla\varphi||_{L^2(\Omega_1\setminus\Omega_{q\varepsilon}}.
\end{equation}

Using (\ref{Nov4b}), (\ref{Nov8aa}) and Corollary \ref{Kor2}, we
arrive at
\begin{kor} Under the assumptions on $\Omega_1$ and $\Omega_2$ of this
section, the following estimate for the eigenvalues of the
problems (\ref{TTN1}) and (\ref{TTN2}) holds
$$
|\lambda_m^{-1}-\mu_k^{-1}|\leq C\max_{\varphi\in X_m,
||\varphi||=1}||\nabla\varphi||_{L^2(\Omega_1\setminus
\Omega_{q\varepsilon})}^2,
$$
where $\mu_1,\ldots,\mu_{J_m}$ are eigenvalues of {\rm
(\ref{TTN2})} located near $\lambda_m^{-1}$, see {\rm Proposition
\ref{Pr1}}.
\end{kor}

\subsection{The case of Lipschitz domains}

We recall the domain $\Omega$ is called Lipschits with Lipschitz
constant less than or equal to $C_*$ if the boundary
$\partial\Omega$ can be covered by a finite number of balls $B$
such that in an appropriate orthogonal system of coordinate
$B\cap\Omega=B\cap\{y=(y',y_n):y_n>h(y')\}$, where $h$ satisfies
$|h(y')-h(z')|\leq C_*|y'-z'|$ and $h(0)=0$ and $B$ has the center
at the origin.

\medskip
We assume in this section that  $\Omega_1$ is a bounded Lipschitz
domain and we denote  by $B_k$, $k=1,\ldots,M$, the balls from the
covering of the boundary and by $h_k$ corresponding Lipschitz
functions. Then there exists a positive $\delta$ such that the set
${\cal V}_\delta=\{x\in\Bbb R^n:\dist
(x,\partial\Omega)\leq\delta\}$ contains in $\bigcup_{k=1}^MB_k$.

Concerning the domain $\Omega_2$ we assume that
$$
\Omega_\varepsilon\subset\Omega_2\subset\Omega_\varepsilon^+,
$$
where $\varepsilon$ is a sufficiently small positive number.
Moreover, we suppose that the domain $\Omega_2$ is Lipschitz and
$$
B_k\cap\Omega_2=B_k\cap\{y=(y',y_n):y_n>g_k(y')\}\;\;\mbox{for
$k=1,\ldots,M$,}
$$
where $g_k$ are Lipschitz functions with the Lipschits constant
$\leq C_*$.

\begin{prop}\label{Prop19a} The following inequality
\begin{equation}\label{Jan15a}
||T_0\varphi||^2\leq
C\int_{\Omega_1\setminus\Omega_2}|\nabla\varphi|^2dx
\end{equation}
holds for all $\varphi\in X_m$. Here $T_0=I-S_0$ and $S_0$ is the
orthogonal projector onto $\W^{1,2}(\Omega_1\cap\Omega_2)$,
compare with Corollary \ref{Kor2}.
\end{prop}
\begin{proof} Since $\varphi\in \W^{1,2}(\Bbb R^n\setminus(\Omega_1\cap\Omega_2))$
the trace of this function on $\partial(\Omega_1\cap\Omega_2)$
belongs to $W^{1/1,2}(\partial(\Omega_1\cap\Omega_2))$ and
$$
||\varphi||_{W^{1/1,2}(\partial(\Omega_1\cap\Omega_2))}\leq
C||\varphi||_{W^{1,2}(\Omega_1\setminus\Omega_2)}.
$$
Since $T_0\varphi$ is harmonic in $\Omega_1\cap\Omega_2$ with the
Dirihlet boundary condition $u=\varphi$ on $\partial\Omega_2$ we
have that
$$
||\nabla T\varphi||_{L^2(\Omega_2)}\leq
c||\varphi||_{W^{1/1,2}(\partial\Omega_2)}\leq
C_1||\nabla\varphi||_{L^2(\Omega_1\setminus\Omega_2)}.
$$
This estimate together with $T\varphi=\varphi$ outside $\Omega_2$
implies (\ref{Jan15a}).
\end{proof}

The estimate obtained in the last proposition together with
Corollary \ref{Kor1} and $T_0\varphi=\varphi$ outside
$\Omega_1\cap\Omega_2$, implies the following

\begin{kor} Let $\Omega_2\subset\Omega_1$. Then for $j=1,\ldots,J_m$,
$$
c\min ||\nabla\varphi||_{L^2(\Omega_1\setminus\Omega_2)}^2\leq
\Big |\mu_{m_j}^{-1}-\lambda_m^{-1}\Big |\leq C\max
||\nabla\varphi||_{L^2(\Omega_1\setminus\Omega_2)}^2,
$$
where minimum and maximum are taken over all $\varphi\in X_m$ with
$||\varphi||=1$.
\end{kor}

Let us turn to estimating of the function $\Psi=\Psi_\varphi$. We
have the following

\begin{prop}\label{Prop19b} There exists a domain $S$ such that $\Omega_2\setminus
\Omega_{1}\subset S\subset\Omega_2$, $|S|\leq
C|(\Omega_2\setminus\Omega_1)|$ and the following inequality
\begin{equation}\label{Jan19b}
||\nabla\Psi_\varphi ||_{L^2(\Omega_2)}^2\leq C_1\max\int_S
|\nabla \varphi|^2dx
\end{equation}
holds, where $\max$ is taken over all $\varphi\in X_m$ satisfying
$||\varphi||=1$. Here by $|S|$ is denote the area of $S$  and the
constant C is independent of $\varepsilon$.

\end{prop}
\begin{proof}
There exists a set of smooth functions $\psi_k\in C^1(B_k)$ with
compact support such that $\psi_1(x)+\cdots+\psi_N(x)=1$ on ${\cal
V}_\delta$.

We take a smooth function $\eta=\eta(t)$ which is equal to $1$ for
$t<3/2$ and to $0$ for $t>2$ and introduce the function
$\zeta_k(x)$ in $B_k$ by
$$
\zeta_k (x)=\eta\Big (\frac{y_n-g_k(y')}{h_k(y')-g_k(y')}\Big )
$$
if $h_k(y')>g_k(y')$ and $\zeta_k (x)=0$ otherwise. Here  the
local variable $y$ is considered as a function of $x$. We define
also
$$
\zeta (x)=\sum_{k=1}^N\zeta_k(x)\psi_k(x).
$$
One can verify that $\zeta(x)=1$ if $x\in \Omega_2$ and
$y_n-g_k(y')<3(h_k(y')-g_k(y'))/2$. Moreover, the following
inequality holds:
\begin{equation}\label{Jan19a}
\int |(\nabla\zeta)u (x)|^2dx\leq C\int_{\supp\zeta} |\nabla
u|^2dx
\end{equation}
for $u\in\W^{1,2}(\Omega_2)$.  It is sufficient to prove a local
version of this inequality, i.e.
\begin{equation}\label{Jan17b}
\int_{B_k} |\nabla_y \zeta_ku|^2dy\leq
c\int_{B_k,0<y_n-g_k(y')<2|h_k(y')-g_k(y')|}|\nabla u|^2dy
\end{equation}
for smooth functions $u$ equals zero for $y_n<g_k(y')$. Estimate
(\ref{Jan17b}) follows from one dimensional Hardy inequality.

Now, we are looking for the function $\Psi$ in the form
$\Psi=\zeta\varphi +v$. Then $v$ satisfies the equation
$$
(v,w)=-\langle \varphi
A_{ij}\partial_{x_j}\zeta,\partial_{x_i}w\rangle+\langle
A_{ij}\partial_{x_j}\zeta\partial_{x_i}\varphi,w\rangle-\lambda_m\langle\zeta\varphi,w\rangle.
$$
Putting here $w=v$ and using (\ref{Jan19a}), we get
$$
(v,v)\leq C\int_{\supp\zeta}|\nabla\varphi|^2dx,
$$
which implies (\ref{Jan19b}).
\end{proof}

From Propositions \ref{Prop19a}, \ref{Prop19b} and Corollary
\ref{Kor2} it follows

\begin{kor} There exists a domain $S$ such that $\Omega_2\setminus
\Omega_{1}\subset S\subset\Omega_2$, $|S|\leq
C|\Omega_2\setminus\Omega_1|$ and the following inequality
$$
\big |\mu_k^{-1}-\lambda_m^{-1}\big |\leq
C_1\max\int_{(\Omega_1\setminus\Omega_2)\cup S} |\nabla
\varphi|^2dx
$$
holds, where $\max$ is taken over all $\varphi\in X_m$ satisfying
$||\varphi||=1$.
\end{kor}

{

\end{document}